%% file: 0.mainbody.tex
\documentclass[11pt,english]{amsart}
\usepackage[latin9]{inputenc}
\usepackage{geometry}
\usepackage{color}
\usepackage{babel}
\usepackage{amsthm}
\usepackage{amssymb}
\usepackage{stackrel}
\usepackage{graphicx}
\usepackage[unicode=true,pdfusetitle,
 bookmarks=true,bookmarksnumbered=false,bookmarksopen=false,
 breaklinks=false,pdfborder={0 0 0},backref=false,colorlinks=true]
 {hyperref}
\hypersetup{
 citecolor=blue}

\makeatletter
\numberwithin{equation}{section}
\numberwithin{figure}{section}
  \theoremstyle{definition}
  \newtheorem{defn}{\protect\definitionname}
  \theoremstyle{plain}
  \newtheorem*{thm*}{\protect\theoremname}
  \theoremstyle{remark}
  \newtheorem*{rem*}{\protect\remarkname}
  \theoremstyle{plain}
  \newtheorem*{prop*}{\protect\propositionname}
  \theoremstyle{plain}
  \newtheorem*{lem*}{\protect\lemmaname}
  \theoremstyle{plain}
  \newtheorem*{cor*}{\protect\corollaryname}
  \theoremstyle{remark}
  \newtheorem*{acknowledgement*}{\protect\acknowledgementname}
\theoremstyle{plain}
\newtheorem{thm}{\protect\theoremname}
  \theoremstyle{remark}
  \newtheorem{rem}{\protect\remarkname}
  \theoremstyle{plain}
  \newtheorem{prop}{\protect\propositionname}
  \theoremstyle{plain}
  \newtheorem{lem}{\protect\lemmaname}
  \theoremstyle{plain}
  \newtheorem{cor}{\protect\corollaryname}
 \theoremstyle{definition}
 \newtheorem*{defn*}{\protect\definitionname}

\address{Department of Mathematics, The University of Chicago, Chicago, IL 60637, USA} 
\email{lkao@math.uchicago.edu}


\subjclass[2000]{}

\thanks{Kao gratefully acknowledges support from the National Science Foundation Postdoctoral Research Fellowship under grant DMS 1703554.}

\usepackage{enumitem}

\def\s{\sigma}
\def\k{\kappa}
\def\t{\tau}

\def\G{\Gamma}

\def\g{\gamma}

\def\L{\Lambda}
\def\w{\omega}
\def\dd{\mathrm{d}}

\def\D{\mathbb{D}}
\def\R{\mathbb{R}}

\def\H{\mathbb{H}}
\def\N{\mathbb{N}}
\def\Z{\mathbb{Z}}

\def\bb{\mathbb}
\def\cal{\mathcal}

\def\psl{\mathrm{PSL}}

\def\vbdy{\partial_{\infty}}

\def\TMS{\Sigma^{+}}

\makeatother

  \providecommand{\acknowledgementname}{Acknowledgement}
  \providecommand{\corollaryname}{Corollary}
  \providecommand{\definitionname}{Definition}
  \providecommand{\lemmaname}{Lemma}
  \providecommand{\propositionname}{Proposition}
  \providecommand{\remarkname}{Remark}
  \providecommand{\theoremname}{Theorem}
\providecommand{\corollaryname}{Corollary}
\providecommand{\theoremname}{Theorem}

\begin{document}

\title{Manhattan Curves For Hyperbolic Surfaces with Cusps}

\author{Lien-Yung Kao}

\maketitle

\begin{abstract}
In this paper, we study an interesting curve, so-called the Manhattan
curve, associated with a pair of boundary-preserving Fuchsian representations
of a (non-compact) surface, especially representations corresponding
to Riemann surfaces with cusps. Using Thermodynamic Formalism (for
countable Markov shifts), we prove the analyticity of the Manhattan
curve. Moreover, we derive several dynamical and geometric rigidity
results, which generalize results of Marc Burger \cite{Burger:1993wb}
and Richard Sharp \cite{Sharp:1998if} for convex-cocompact Fuchsian
representations.
\end{abstract}

\tableofcontents{}

\input{1.introduction.tex}

\input{2.preliminaries.tex}

\input{3.Extended_Schottky.tex}

\input{4.Manhattan_Curve.tex}

\input{5.Appendix.tex}

\bibliographystyle{amsalpha}
\bibliography{/Users/nyima/Dropbox/TEX/Bibtex/Papers3_backup_20160901,Papers3_backup_20160901}

\end{document}

%% file: 1.introduction.tex
\section{Introduction}

This paper is devoted to studying relations between Fuchsian representations
of a (non-compact) surface through a dynamics tool, namely, Thermodynamic
Formalism (for countable Markov shifts). Using a symbolic dynamics
model associated with these representations, we investigate several
closely related and informative geometric and dynamical objects arising
from them, such as the critical exponent, the Manhattan curve, and
Thurston's intersection number. For dynamics, we prove a version of
the famous Bowen's formula, which characterizes several geometric
and dynamics quantities via the (Gurevich) pressure. Moreover, we
analyze the phase transition of the pressure function (of weighted
geometric potentials) in detail; thus, we have a control of the analyticity
of the pressure. In geometry, we recover and extend several rigidity
results, such as Bishop-Steger entropy rigidity and Thurston's intersection
number rigidity, to Riemann surfaces of infinite volume and with cusps. 

To put our results in context, we shall start from notations and definitions.
Throughout the paper, $S$ denotes a (topological) surface with negative
Euler characteristic. Let $\rho_{1},\rho_{2}$ be two Fuchsian (i.e.,
discrete and faithful) representations of $G:=\pi_{1}S$ into ${\rm PSL}(2,\R)$.
For short, we denote $\rho_{i}(G)$ by $\G_{i}$, by $S_{i}=\G_{i}\backslash\H$
the Riemann surface of $\rho_{i}$ for $i=1,2$. We write $h_{top}(S_{1})$
and $h_{top}(S_{2})$ for the \textit{topological entropy of the geodesic
flow} for $S_{1}$ and $S_{2}$.

The group $G$ acts diagonally on $\H\times\H$ by $\g\cdot x=(\rho_{1}(\g)x_{1},\rho_{2}(\g)x_{2})$
where $x=(x_{1},x_{2})\in\H\times\H$ and $\g\in G$. We are interested
in \textit{weighted Manhattan metrics} $d_{\rho_{1},\rho_{2}}^{a,b}$
associated with $S_{1}$ and $S_{2}$. More precisely, fix $o=(o_{1},o_{2})\in\H\times\H$,
$d_{\rho_{1},\rho_{2}}^{a,b}(o,\gamma o):=a\cdot d(o_{1},\rho_{1}(\gamma)o_{1})+b\cdot d(o_{2},\rho_{2}(\g)o_{2})$.
Moreover, we always assume that $a,b\geq0$ and $a,b$ do not vanish
at the same time, i.e., throughout this paper we assume that $(a,b)\in D:=\{(x,y)\in\R^{2}:\ x\geq0,y\geq0\}\backslash(0,0)$.

\begin{defn}
The Poincaré series of the weighted Manhattan metric $d_{\rho_{1},\rho_{2}}^{a,b}$
is defined as 
\[
Q_{\rho_{1},\rho_{2}}^{a,b}(s)={\displaystyle \sum_{\g\in G}\exp(-s\cdot d_{\rho_{1},\rho_{2}}^{a,b}(o,\g o)).}
\]
Moreover, $\delta_{\rho_{1},\rho_{2}}^{a,b}$ denotes the critical
exponent of $Q_{\rho_{1},\rho_{2}}^{a,b}(s)$, that is, $Q_{\rho_{1},\rho_{2}}^{a,b}(s)$
diverges when $s<\delta_{\rho_{1},\rho_{2}}^{a,b}$ and $Q_{\rho_{1},\rho_{2}}^{a,b}(s)$
converges when $s>\delta_{\rho_{1},\rho_{2}}^{a,b}$. For short, if
there is no confusion, we will always drop the subscripts $\rho_{1},\rho_{2}$.
\end{defn}

Noticing that the critical exponent $\delta^{a,b}$, by the triangle
inequality, is independent on the choice of the reference point $o=(o_{1},o_{2})$.
We remark that when $a=0$ (or $b=0$), we are back to the classical
critical exponent of $\rho_{1}(G)$ (or $\rho_{2}(G)$), and by Sullivan's
result we know that $\delta^{1,0}=h_{top}(S_{1})$ and $\delta^{0,1}=h_{top}(S_{2})$.

\begin{defn}
[The Manhattan Curve] \label{def:-The-Manhattan-Curve}The Manhattan
curve ${\cal C}=\mathcal{C}(\rho_{1},\rho_{2})$ of $\rho_{1}$, $\rho_{2}$
is the boundary of the set 
\[
\{(a,b)\in\R^{2}:\ {\displaystyle Q_{\rho_{1},\rho_{2}}^{a,b}(1)<\infty\}.}
\]
Alternatively, ${\cal C}$ can be defined as
\[
\{(a,b)\in\R^{2}:\ {\displaystyle Q_{\rho_{1},\rho_{2}}^{a,b}(s)\mbox{ has critical exponent }1\}.}
\]

\end{defn}

Our first result gives a rough picture of the Manhattan curve $\mathcal{C}(\rho_{1},\rho_{2})$
of $\rho_{1}$ and $\rho_{2}$ for any pair of Fuchsian representations.

\begin{thm*}
[Theorem \ref{thm: Manhattan curves are conti}]Let $S$ be a (topological)
surface with negative Euler characteristic, and let $\rho_{1},\rho_{2}$
be two Fuchsian representations of $G:=\pi_{1}S$ into ${\rm PSL}(2,\R)$.
We denote $S_{1}=\rho_{1}(G)\backslash\H$ and $S_{2}=\rho_{2}(G)\backslash\H$.
Then

\begin{itemize}[font=\normalfont] 

\item[(1)] $(h_{top}(S_{1}),0)$ and $(0,h_{top}(S_{2}))$ are on
$\mathcal{C}$; 

\item[(2)] $\mathcal{C}$ is convex; and

\item[(3)] $\mathcal{C}$ is continuous.

\end{itemize}
\end{thm*}

Let us briefly review the history of the Manhattan curve ${\cal C}(\rho_{1},\rho_{2})$.
In \cite{Burger:1993wb}, using the Patterson-Sullivan argument, Burger
proved that for $\rho_{1}$ and $\rho_{2}$ are convex co-compact
(i.e., both $\rho_{1}(G)$ and $\rho_{2}(G)$ have no parabolic element),
one has $\mathcal{C}$ is $C^{1}$. In \cite{Sharp:1998if}, Sharp
employed Thermodynamic Formalism to prove that ${\cal C}$ is real
analytic. In this work, we are interested in representations which
are not convex co-compact. The presence of parabolic elements greatly
complicates the problem. Nevertheless, thanks to recent developments
on Thermodynamic Formalism for countable Markov shifts, we are able
to generlize these results to surfaces with cusps.

We mainly work on representations that satisfy the following two geometric
conditions, namely, being \textit{boundary-preserving isomorphic}
and the \textit{extended Schottky condition}. 

Two Fuchsian representations $\rho_{1}$ and $\rho_{2}$ are \textit{boundary-preserving
isomorphic} if there exists an isomorphism $\iota:\rho_{1}(G)\to\rho_{2}(G)$
so that $\iota$ is \textit{type-preserving} and \textit{peripheral-structure-preserving}.
More precisely, $\iota$ is type-preserving if $\iota$ sends parabolic
elements to parabolic elements and hyperbolic elements to hyperbolic
elements, and $\iota$ is peripheral-structure-preserving if for any
element $\g\in\rho_{1}(G)$ corresponding to a geodesic boundary of
$S_{1}$, its image $\iota(\g)$ corresponds to a geodesic boundary
of $S_{2}$ and vise versa. 

We say a Riemann surface $\rho(G)\backslash\H$ is a \textit{extended
Schottky surface} or $\rho$ satisfies the \textit{extended Schottky
condition} if $\rho$ satisfies ${\rm (C1),\ (C2),\ (C3)}$, and $N_{1}+N_{2}\geq3$
(see Definition \ref{def:extended Schottky}). Roughly speaking, an
extended Schottky surface is a geometrically finite Riemann surface
with cusps, funnels or both ends and whose group deck transformations
is a free group. Extended Schottky surfaces with no cusp are called
\textit{classical Schottky surfaces,} and they are known to be convex
co-compact. One example of extended Schottky surfaces is the surface
with two cusps and two funnels. 
\begin{rem*}
$\ $\begin{itemize}[font=\normalfont] 

\item[(1)] For $\rho_{1}$, $\rho_{2}$ two convex co-compact Fuchsian
representations, $\rho_{1}$ and $\rho_{2}$ are always type-preserving
isomorphic (because they have no parabolic element). However, it does
not guarantee that $S_{1}$ and $S_{2}$ are homeomorphic, for example,
one holed torus and a pair of pants. Therefore, the peripheral-structure-preserving
condition is necessary to derive a homeomorphism between $S_{1}$
and $S_{2}$ (see Theorem \ref{thm: bilipsthitz } for more details).

\item[(2)] The extended Schottky condition that we use here was introduced
in Dal'Bo-Peigné \cite{Dalbo:1996wh}. This condition is needed in
our argument for some technical reasons.

\end{itemize}
\end{rem*}
Now, we are ready to present our main results. Let $\rho_{1},\rho_{2}$
be two boundary-preserving isomorphic Fuchsian representations satisfying
the extended Schottky condition. For the convenience of presentation,
we leave precise definitions of many dynamics and geometry terminologies
in Section \ref{sec:Preliminaries-on-Thermodynamic}.

Following Dal'bo-Peigné \cite{Dalbo:1996wh}, there exists a symbolic
coding of closed geodesics on extended Schottky surfaces. Here we
summarize relevant results in \cite{Dalbo:1996wh}.

\begin{prop*}
[Propsition \ref{prop:coding_property}, Propsition \ref{thm:coding of geodesic flow}, Lemma \ref{lem:roof fcn >0}]There
exists a topologically mixing countable Markov shift $(\TMS,\s)$
satisfying the BIP property. Moreover, there is a function $\tau:\TMS\to\R^{+}$
(resp. $\kappa:\TMS\to\R^{+})$ such that all but finitely many closed
geodesics on $S_{1}$ (resp. $S_{2}$) are coded by ${\rm Fix}(\TMS)$
the fixed points of $\sigma$ and lengths of these closed geodesics
are given by $\tau$ (reps. $\kappa$). Furthermore, $\tau$ and $\kappa$
are locally Hölder and bounded away from zero. 
\end{prop*}

Because $\tau$ and $\kappa$ are constructed by the geometric potential
of the corresponding Bowen-Series map on the boundary of $T^{1}S_{1}$
and $T^{1}S_{2}$, we will continue calling them by geometric potentials
(see Section \ref{sec:Preliminaries-on-Extended} for more details). 

The following lemma is one of the most important result of this work.
Recall that for a finite Markov shift, the (Gurevich) pressure $P_{\sigma}$
has no phase transition, that is, the pressure function $t\mapsto P_{\sigma}(tf)$
is analytic for $f$ a Hölder continuous potential. Whereas, for countable
Markov shifts, Sarig \cite{Sarig:1999wo,Sarig:2001bj} and Mauldin-Urba{\'n}ski
\cite{Mauldin:2003dn} pointed out that, for $f$ a locally Hölder
continuous potential, $t\mapsto P_{\sigma}(tf)$ is not analytic.
Inspired by the work of Iommi-Riquelme-Velozo \cite{Iommi:2016tv},
we study the phase transition in detail and give a precise picture
of the pressure function of weighted geometric potentials.

\begin{lem*}
[Lemma \ref{lem:compute critical exponent}, Lemma \ref{lem:phas transition and critical exponent}]Let
$(\TMS,\sigma)$ be the countable Markov shift and $\tau$, $\kappa$
be the geometric potentials given by the above proposition. We have,
for $a,b\geq0,$ 
\[
P_{\sigma}(-t(a\tau+b\k))=\begin{cases}
{\rm infinite,} & {\rm for\ }t<\frac{1}{2(a+b)};\\
{\rm real\ analytic,} & {\rm {\rm for\ }}t>\frac{1}{2(a+b)}.
\end{cases}
\]

\end{lem*}

Furthermore, similar to Bowen's formula for hyperbolic flows over
compact metric spaces, we give a geometric interpretation of the solution
for the equation $P_{\s}(tf)=0$ when $f$ is a weighted geometric
potential. Namely, we prove that the critical exponent $\delta^{a,b}$
can be realized by the growth rate of hyperbolic elements (or equivalently,
closed geodesics).

\begin{thm*}
[Bowen's Formula; Lemma \ref{lem:P(-at-bk)=00003D0} , Theorem \ref{thm:P(-at-bk)=00003D0 analytic}, Theorem \ref{thm:t_ab=00003Ddelta_ab}]The
set $\{(a,b)\in D:\ P_{\sigma}(-a\tau-b\kappa)=0\}$ is a real analytic
curve. Moreover, for each $(a,b)\in D$ there exists a unique $t_{a,b}$
such that 
\[
P_{\sigma}(-t_{a,b}(a\t+b\k))=0.
\]
 Furthermore, $t_{a,b}=\delta^{a,b}$. 
\end{thm*}

Combing the above theorems, we have the following results for the
Manhattan curve ${\cal C}(\rho_{1},\rho_{2})$.

\begin{thm*}
[Theorem \ref{Thm: Manhattan cuve analytic}] ${\cal C}(\rho_{1},\rho_{2})$
is real analytic.
\end{thm*}

Moreover, using the analyticity of the Manhattan curve and the uniqueness
of the equilibrium states, we have better picture of the Manhattan
curve ${\cal C}(\rho_{1},\rho_{2})$.

\begin{prop*}
[Proposition \ref{prop: Manhattan curve features}] We have

\begin{itemize}[font=\normalfont] 

\item[(1)] ${\cal C}(\rho_{1},\rho_{2})$ is strictly convex if $\rho_{1}$
and $\rho_{2}$ are NOT conjugate in ${\rm PSL}(2,\R)$; and

\item[(2)] ${\cal C}(\rho_{1},\rho_{2})$ is a straight line if and
only if $\rho_{1}$ and $\rho_{2}$ are conjugate in ${\rm PSL}(2,\R)$.

\end{itemize}
\end{prop*}

\begin{rem*}
Using Paulin-Pollicott-Schapira's arguments in \cite{Pollicott:2012ud},
as well as Dal'Bo-Kim's Patterson-Sullivan theory approach in \cite{Dalbo:2008dp},
it is possible to recover some of above results without using symbolic
dynamics. However, due to the author's limited knowledge, without
using symbolic dynamics, there seems no clear path to proving the
analyticity of the Manhattan curve ${\cal C}(\rho_{1},\rho_{2})$.
\end{rem*}

Furthermore, we have the following rigidity corollaries.

\begin{cor*}
[Bishop-Steger's entropy rigidity; cf. \cite{Bishop:1991gz}; Corollary \ref{cor:-Bishop-Steger}]We
have, for any $o\in\H$, 
\[
\delta^{1,1}=\lim_{T\to\infty}\frac{1}{T}\log\#\{\g\in G:\ d(o,\rho_{1}(\g)o)+d(o,\rho_{2}(\g)o)\leq T\}.
\]
 Moreover, $\delta^{1,1}\leq\dfrac{h_{top}(S_{1})\cdot h_{top}(S_{2})}{h_{top}(S_{1})+h_{top}(S_{2})}$
and the equality holds if and only if $S_{1}$ and $S_{2}$ are isometric.
\end{cor*}

\begin{rem*}
In Bishop-Steger's paper \cite{Bishop:1991gz}, their result holds
for finite volume Fuchsian representations (i.e., lattices). We extend
this result to some infinite volume Fuchsian representations.
\end{rem*}

\begin{defn}
[Thurston's Intersection Number] \label{def: Thurston's-intersection-nb}Let
$S_{1}$ and $S_{2}$ be two Riemann surfaces. Thurston's intersection
number $I(S_{1},S_{2})$ of $S_{1}$ and $S_{2}$ is given by 
\[
{\rm I}(S_{1},S_{2})=\lim_{n\to\infty}\frac{l_{2}[\g_{n}]}{l_{1}[\g_{n}]}
\]
 where $\{[\g_{n}]\}_{n=1}^{\infty}$ is a sequence of conjugacy classes
for which the associated closed geodesics $\g_{n}$ become equidistributed
on $\G_{1}\backslash\H$ with respect to area.
\end{defn}

\begin{cor*}
[Thurston's Rigidity; cf. \cite{Thurston:1998wp}; Corollary \ref{cor:Thurston's rigidity}]
Let $\rho_{1},\rho_{2}$ be two boundary-preserving isomorphic Fuchsian
representations satisfying the extended Schottky condition, and $S_{1}=\rho_{1}(G)\backslash\H$,
$S_{2}=\rho_{2}(G)\backslash\H$. Then ${\rm I}(S_{1},S_{2})\geq\frac{h_{top}(S_{1})}{h_{top}(S_{2})}$
and the equality hold if and only if $\rho_{1}$ and $\rho_{2}$ are
conjugate in ${\rm PSL}(2,\R)$.
\end{cor*}

The outline of this paper is as follows. In Section \ref{sec:Preliminaries-on-Thermodynamic},
we briefly review necessary background of Thermodynamic Formalism
(for countable Markov shifts) and hyperbolic geometry. In Section
\ref{sec:Preliminaries-on-Extended}, we introduce extended Schottky
surfaces. Moreover, we study the phase transition of the geodesic
flows on them, which is one of the most important results in this
work. Section \ref{sec:The-Manhattan-Curve} is devoted to the proof
of our main results. Using Paulin-Pollicott-Schapira's arguments in
\cite{Pollicott:2012ud}, we derive geometric interpretations of the
critical exponent $\delta^{a,b}$ and, thus, we are able to link it
with the (symbolic) suspension flow and Bowen's formula.

\begin{acknowledgement*}
The author is extremely grateful to his Ph.D advisor Prof. François
Ledrappier. Nothing would have been possible without François\textquoteright{}
support, guidance, especially those he gave the author after his retirement.
The author also appreciates Prof. Richard Canary for his numerous
helps on geometry parts of this work, and Prof. Godofredo Iommi for
his wonderful explanations on his works. Lastly, the author would
like to thank Prof. Françoise Dal'Bo and Prof. Marc Peigné for their
helpful comments of this work, and Dr. Felipe Riquelme for pointing
out some errors in the earlier version.
\end{acknowledgement*}

%% file: 2.preliminaries.tex
\section{Preliminaries \label{sec:Preliminaries-on-Thermodynamic}}

\subsection{Thermodynamic Formalism for Countable Markov Shifts}

Let $\mathcal{S}$ be a countable set and $\mathbb{A}=(t_{ab})_{\mathcal{S}\times\mathcal{S}}$
be a matrix of zeroes and ones indexed by $\mathcal{\mathcal{S}}\times\mathcal{S}$.

\begin{defn}
The \textit{one-sided (countable) Markov shift} $(\TMS_{\mathbb{A}},\sigma)$
with the set of \textit{alphabet }$\mathcal{S}$ is the set 
\[
\TMS_{\mathbb{A}}=\{x=(x_{n})\in\mathcal{S}^{\bb N}:\ t_{x_{n}x_{n+1}}=1\ \mbox{for every }n\in\bb N\}
\]
 coupled with the (left) shift map $\sigma:\TMS_{\mathbb{A}}\to\TMS_{\mathbb{A}}$,
$(\sigma(x))_{i}=(x)_{i+1}$.
\end{defn}
We will alway drop the subscript $\mathbb{A}$ of $\TMS_{\mathbb{A}}$
when there is no ambiguity on the adjacency matrix $\mathbb{A}$.
Furthermore, we endow $\TMS$ with the relative product topology,
which is given by the base of \textit{cylinders} 
\[
[a_{0},...,a_{n-1}]:=\{x\in\TMS:\ a_{i}=x_{i},\ \mbox{for }\mbox{ }0\leq i\leq n-1\}.
\]
A \textit{word} on an alphabet $\mathcal{S}$ is an element $(a_{0},a_{2},...,a_{n-1})\in\mathcal{S}^{n}$
($n\in\bb N)$. The \textit{length} of the word $(a_{0},a_{2},...,a_{n-1})$
is $n$. A word is called \textit{admissible} (w.r.t. an adjacency
matrix $\mathbb{A}$) if the cylinder it defines is non-empty. 

In the following, we will assume $(\TMS,\sigma)$ is \textit{topologically
mixing}, that is, for any $a,b\in{\cal {\cal S}}$, there exists an
$N\in\N$ such that $\s^{-n}[a]\cap[b]$ is non-empty for all $n>N$.
Noticing that under the topologically mixing assumption and the BIP
property below, the thermodynamics formalism for countable Markov
shifts is well-studied and very close to the classical thermodynamic
formalism for finite Markov shifts. 

The \textit{$n$-th variation} of a function $g:\TMS\to\R$ is defined
by 
\[
V_{n}(g)=\sup\{|g(x)-g(y)|:\ x,y\in\TMS,x_{i}=y_{i}\ \mbox{for}\ i=1,2,...,n\}.
\]
We say $g$ has \textit{summable variation} if $\sum_{n=1}^{\infty}V_{n}(g)<\infty$,
and $g$ is \textit{locally Hölder} if there exists $c>0$ and $\theta\in(0,1)$
such that $V_{n}(g)\leq c\theta^{n}$ for all $n\geq1$.

\begin{defn}
[Gurevich Pressure for Markov Shifts] Let $g:\TMS\to\R$ have summable
variation. The \textit{Gurevich pressure} of $g$ is defined by 
\[
P_{\sigma}(g)=\lim_{n\to\infty}\frac{1}{n}\log\sum_{x\in{\rm {Fix}}^{n}}e^{S_{n}g(x)}\chi_{[a]}(x)
\]
where ${\rm {Fix}}^{n}:=\{x\in\Sigma^{+}:\ \sigma^{n}x=x\}$ and $a$
is any element of $\mathcal{\mathcal{S}}$ and $S_{n}g(x):=\sum_{i=0}^{n-1}g(\sigma^{i}x)$.
\end{defn}

As pointed out by Sarig (cf. Theorem 1 \cite{Sarig:1999wo}) that
the limit exists and is independent of the choice of $a\in\mathcal{\mathcal{\mathcal{S}}}$.

\begin{thm}
[Variational Principle; Theorem 3 \cite{Sarig:1999wo}]\label{thm:variational principle for shift}
Let $(\TMS,\sigma)$ be a topologically mixing countable Markov shift
and $g$ have summable variation. If $\sup g<\infty$ then 
\[
P_{\sigma}(g)=\sup\left\{ h_{\sigma}(\mu)+\int_{\TMS}g\dd\mu:\ \mu\in\mathcal{M}_{\sigma}\ \mbox{and}\ -\int_{\TMS}g\dd\mu<\infty\right\} ,
\]
 where $\mathcal{M}_{\sigma}$ is the set of $\sigma-$invariant Borel
probability measures on $\TMS$. 
\end{thm}

For $\mu\in{\cal M}_{\sigma}$ such that $P_{\sigma}(g)=h_{\sigma}(\mu)+\int_{\TMS}g\dd\mu$,
we call such a measure $\mu$ an \textit{equilibrium state} for the
function $g$.

\begin{defn}
[BIP] \label{def:BIP}A (countable) Markov shift $(\TMS_{\mathbb{A}},\sigma)$
has the BIP (Big Images and Preimages) property if and only if there
exists $\{b_{1},b_{2},...,b_{n}\}\subset\bb N$ such that for every
$a\in\bb N$ there exists $i,j\in\bb N$ with $t_{b_{j}a}t_{ab_{j}}=1$.
\end{defn}

The following theorem about the analyticity of pressure is found independently
by Mauldin-Urba{\'n}ski \cite{Mauldin:2003dn} and Sarig \cite{Sarig:2003hl}.
There are minor differences between their original statements; however,
under the topologically mixing and the BIP assumptions their results
are the same (see Remark \ref{rem: MU=00003DSarig with BIP} for more
details).

\begin{thm}
[Analyticitly of Pressure; Theorem 2.6.12, 2.6.13 \cite{Mauldin:2003dn}, Corollary 4 \cite{Sarig:2003hl}]\label{thm:Analyticty pressure}Let
$(\TMS,\sigma)$ be a topologically mixing countable Markov shift
with the BIP property. If $\Delta\subset\R$ is a interval and $t\to f_{t}$
a real analytic family of locally Hölder continuous functions with
$P_{\sigma}(f_{t})<\infty$, then $t\to P_{\sigma}(f_{t})\in\R$,
$t\in\Delta$, is also real analytic. Moreover, the derivative of
the pressure function is 
\[
\left.\frac{d}{dt}P_{\s}(f_{t})\right|_{t=0}=\int_{\TMS}f_{0}\dd\mu_{f_{0}}
\]
 where $\mu_{f_{0}}$ is the equilibrium state for $f_{0}$.
\end{thm}

\begin{rem}
\label{rem: MU=00003DSarig with BIP}$\ $\begin{enumerate}[font=\normalfont] 

\item We combine Proposition 2.1.9 and Theorem 2.6.12 in \cite{Mauldin:2003dn}
in the following way to derive Theorem \ref{thm:Analyticty pressure}.
By Proposition 2.1.9, we know that $P_{\sigma}(f_{t})<\infty$ implies
$f_{t}$ are summable Hölder functions (i.e., $f_{t}\in\mathcal{K}_{\beta}^{s}$
in \cite{Mauldin:2003dn} notation). The rest is a direct consequence
of Theorem 2.6.12.

\item A topologically mixing countable Markov shift $(\TMS,\sigma)$
with the BIP property is indeed a graph directed Markov system with
a \textit{finitely irreducible} adjacency matrix defined in \cite{Mauldin:2003dn}.
Hence the definition of (Gurevich) pressure given in here (from Sarig
\cite{Sarig:1999wo}) matches with the one given in Mauldin-Urba{\'n}ski
\cite{Mauldin:2003dn} (cf. Section 7 \cite{Mauldin:2001dn}).

\item For Corollary 4 in \cite{Sarig:2003hl}, it requires $f_{t}$
to be \textit{positive recurrent}. However, under the same assumptions
as in Theorem \ref{thm:Analyticty pressure} (i.e., $(\TMS,\sigma)$
is topologically mixing with the BIP property and $f_{t}$ are functions
of summable variation with $P_{\sigma}(f_{t})<\infty$) then one can
prove $f_{t}$ are positive recurrent (cf. Corollary 2 \cite{Sarig:2003hl}
or Proposition 3.8 \cite{Sarig:2009wta}).

\end{enumerate}
\end{rem}

\begin{thm}
[Phase Transition; \cite{Sarig:1999wo,Sarig:2001bj}, \cite{Mauldin:2003dn}]\label{thm: Phase transition}Let
$(\TMS,\sigma)$ be a countable Markov shift with the BIP property
and $g:\TMS\to\R$ be a positive locally Hölder continuous function.
Then there exists $s_{\infty}>0$ such that the pressure function
$t\to P_{\sigma}(-tg)$ has the following properties
\[
P_{\sigma}(-tg)=\begin{cases}
\infty & \mbox{if}\ t<s_{\infty},\\
\mbox{real analytic} & \mbox{if}\ t>s_{\infty}.
\end{cases}
\]
 Moreover, if $t>s_{\infty}$ there exists a unique equilibrium state
for $-tg$.
\end{thm}

Recall that two functions $f,g:\TMS\to\R$ are said to be \textit{cohomologous},
denoted by $f\sim g$, via a \textit{transfer function} $h$, if $f=g+h-h\circ\sigma$.
A function which is cohomologous to zero is called a\textit{ coboundary}.

\begin{thm}
[Liv\v{s}ic Theorem; Theorem 1.1 \cite{Sarig:2009wta}] \label{thm:Livsic}Suppose
$(\TMS,\sigma)$ is topologically mixing, and $f,g:\TMS\to\R$ have
summable variation. Then $f$ and $g$ are cohomologous if and only
if for all $x\in\TMS$ and $n\in\N$ such that $\sigma^{n}(x)=x$,
$S_{n}f(x)=S_{n}g(x)$. 
\end{thm}

\subsection{Thermodynamic Formalism for Suspension Flows }

Let $(\TMS,\sigma)$ be a topologically mixing (countable) Markov
shift and $\tau:\TMS\to\R^{+}$ be a positive function of summable
variation and bounded away from zero which we call the \textit{roof
function}. We define the \textit{suspension space} (relative to $\tau)$
as 
\[
\TMS_{\tau}:=\{(x,t)\in\TMS\times\R:\ 0\leq t\leq\tau(x)\},
\]
 with the identification $(x,\tau(x))=(\sigma x,0)$.

The \textit{suspension flow} $\phi$ (relative to $\tau)$ is defined
as the (vertical) translation flow on $\TMS_{\tau}$ given by 
\[
\phi_{t}(x,s)=(x,s+t)\ \mbox{for}\ 0\leq s+t\leq\tau(x).
\]

Let $F:\TMS_{\tau}\to\R$ be a continuous function, we define $\Delta_{F}:\TMS\to\R$
as 
\[
\Delta_{F}(x)=\int_{0}^{\tau(x)}F(x,t)\dd t.
\]

The following version of the Gurevich pressure for suspension flows
is given in Kempton \cite{Kempton:2011hs}.

\begin{defn}
[Gurevich Pressure for Suspension Flows]\label{def:-Gurevich pressure for sus flow}
Suppose $F:\TMS_{\tau}\to\R$ is a function such that $\Delta_{F}:\TMS\to\R$
has summable variation. The \textit{Gurevich pressure} of $F$ over
the suspension flow $(\TMS_{\tau},\phi)$ is defined as 
\[
P_{\phi}(F):=\lim_{T\to\infty}\frac{1}{T}\log\left(\sum_{\underset{0\leq s\leq T}{\phi_{s}(x,0)=(x,0)}}\exp\left(\int_{0}^{s}F(\phi_{t}(x,0))dt\right)\chi_{[a]}(x)\right),
\]
where $a$ is any element of $\mathcal{S}$.
\end{defn}

Notice that as pointed out by Kempton (cf. Lemma 3.3 \cite{Kempton:2011hs}),
this definition is independent with the choice of $a\in\mathcal{S}$.
Moreover, there are several alternative ways of defining the Gurevich
pressure for suspension flows such as using the variational principle.
In the following, we summarize some of them from works of Savchenko
\cite{Savchenko:1998fh}, Barreira-Iommi \cite{Barreira:2006fd},
Kempton \cite{Kempton:2011hs}, and Jaerisch-Kesseböhmer-Lamei \cite{Jaerisch:2014js}.

\begin{thm}
[Charaterizations for the Gurevich Pressure] \label{thm:Pressure formuals for flow}
Under the same assumptions as in Definition \ref{def:-Gurevich pressure for sus flow},
we have: 
\begin{alignat*}{1}
P_{\phi}(F) & =\inf\{t\in\R:\ P_{\sigma}(\Delta_{F}-t\tau)\leq0\}\\
 & =\sup\{t\in\R:\ P_{\s}(\Delta_{F}-t\tau)\geq0\}\\
 & =\sup\left\{ h_{\phi}(\nu)+\int_{\TMS_{\tau}}F\dd\nu:\ \nu\in\mathcal{M_{\phi}}\ \mbox{and }-\int_{\TMS_{\tau}}\tau\dd\nu<\infty\right\} ,
\end{alignat*}

where $\mathcal{M}_{\phi}$ is the set of $\phi-$invariant Borel
probability measures on $\TMS_{\tau}$. 
\end{thm}

As before, we call a measure $\nu\in{\cal M}_{\phi}$ an \textit{equilibrium
state} for $F$ if $P_{\phi}(F)=h_{\phi}(\nu)+\int F\dd\nu$.

\subsection{Hyperbolic Surfaces}

Let $S$ be a surface with negative Euler characteristic. Recall that
a \textit{Fuchsian representation} $\rho$ is a discrete and faithful
representation from $G:=\pi_{1}S$ to $\rho(G):=\G\leq$${\rm PSL}(2,\R)\cong{\rm Isom}(\H)$.
It is well-known that all hyperbolic surfaces (i.e., surfaces with
constant Gaussian curvature $-1$) can be realized by a Fuchsian representation,
and vise versa. A Fuchsian representation is called \textit{geometrically
finite} if there exists a fundamental domain which is a finite-sided
convex polygon. Recall that $\vbdy\H$ the boundary of $\H$ is defined
as $\R\cup\{\infty\}$, and the \textit{limit set} $\Lambda(\G)\subset\vbdy\H$
of $\G$ is the set of limit points of all $\G$-orbits $\G\cdot o$
for $o\in\H$. We call an element $\g\in\G$ \textit{hyperbolic} (reps.
\textit{parabolic}), if $\g$ has exactly two (resp. one) fixed points
on $\vbdy\H$. For a hyperbolic element $\g$ we denote the \textit{attracting
fixed point} by $\g^{+}$ (i.e., $\g^{+}=\lim_{n\to\infty}\g^{n}o$)
and \textit{repelling fixed point} by $\g^{-}$ (i.e., $\g^{-}=\lim_{n\to\infty}\g^{-n}o$).
For each hyperbolic element $\g\in\G$, the geodesic on $\H$ connecting
$\g^{-}$ and $\g^{+}$ projects to a closed geodesic on $\G\backslash\H$.
We denote this closed geodesic on $\G\backslash\H$ by $\lambda_{\gamma}$.
Conversely, each closed geodesic $\lambda$ on $\G\backslash\H$ it
corresponds to a unique hyperbolic element (up to conjugation) which
is denoted by $\gamma_{\lambda}$. Moreover, the length $l[\lambda_{\g}]$
of the closed geodesic $\lambda_{\g}$ is exactly the translation
distance $l[\g]$ of $\g$, where $l[\g]:=\min\{d(x,\g x):\ x\in\H\}$.

\begin{defn}
The \textit{Busemann function} $B:\vbdy\H\times\H\times\H$ is defined
as
\[
B_{\xi}(x,y):=\lim_{z\to\xi}d(x,z)-d(x,y)
\]
 where $\xi\in\vbdy\H$ and $x,y,z\in\H$. 
\end{defn}

We summarize several well-known properties of the Busemann function:

\begin{prop}
Let $B:\vbdy\H\times\H\times\H\to\R$ be the Busemann function. Then
for $\xi\in\vbdy\H$ and $x,y,z\in\H$ \begin{enumerate}[font=\normalfont]

\item $B_{\xi}(x,y)+B_{\xi}(y,z)=B_{\xi}(x,z)$;

\item For $\gamma\in{\rm PSL}(2,\R)$, $B_{\g(\xi)}(\g(x),\g(y))=B_{\xi}(x,y)$;
and

\item $B_{\xi}(x,y)\leq d(x,y)$.

\end{enumerate}
\end{prop}

\begin{rem}
$\ $\begin{enumerate}[font=\normalfont]

\item Equivalently, using the Poincaré disk model, we can replace
$\H$ by the unit disk $\D$ (through the map $\Psi:\H\to\D$ where
$\mbox{\ensuremath{\Psi}}(z)=i\frac{z-i}{z+i}$). We have ${\rm Isom}(\H)\cong{\rm Isom}(\D)\cong\psl(2,\R)$.
In this paper, we will alternate the use of $\H$ and $\D$ depending
on the convenience of computation and presentation. 

\item In the Poincaré disk model, $\vbdy\D$ is $S^{1}$ and the
Busemann function $B:\vbdy\D^{1}\times\D\times\D\to\R$ satisfies
the same properties stated above. 

\item Moreover, there is a neat formula for the Busemann function:
for $\xi\in\vbdy\D$
\[
\left|\g'(\xi)\right|=e^{B_{\xi}(o,\g^{-1}o)}
\]
where $\g(z):\D\to\D$ is the Möbius map associated with $\g\in{\rm PSL}(2,\R)$
and $o$ is the origin.

\end{enumerate}

\end{rem}

\subsubsection{Marked Length Spectrum }

As mentioned in the previous subsection, for a hyperbolic surface
$R=\G\backslash\H$, there exists a bijection between free homotopy
classes on $R$ and conjugacy classes of $\G$. Moreover, we have
a bijection between closed geodesics on $R$ and conjugacy classes
of hyperbolic elements of $\G$.

\begin{defn}
A \textit{marked length spectrum} function $l:[c]\mapsto l[c]\in\R^{+}$
which assigns to a homotopy class $[c]$ the length $l[c]$. In other
words, it is also the function $l:[h]\mapsto l[h]$ which assigns
to a conjugacy class of a hyperbolic element $[h]$ of the length
$l[h]$ of the corresponding unique closed geodesic.
\end{defn}

The following theorem shows that for each Fuchsian representation
its proportional marked length spectrum determines the surface. We
remark that for convex-cocompact cases the same result was stated
(without a proof) in Burger \cite{Burger:1993wb}. For general Fuchsian
representations, we found it in \cite{Kim:2001km}.

\begin{thm}
[Proportional Marked Length Spectrum Rigidity; Theorem A \cite{Kim:2001km} ]\label{thm:proportinal marked length spectrum}Let
$\rho_{1},\rho_{2}:G\to{\rm PSL}(2,\R)$ be Zariski dense Fuchsian
representations having the proportional marked length spectrum (i.e.,
there exists a constant $c>0$ such that $l[\rho_{1}(\gamma)]=c\cdot l[\rho_{2}(\gamma)]$
for all $\g\in G)$. Then $\rho_{1}$ and $\rho_{2}$ are conjugate
in ${\rm PSL}(2,\R)$. 
\end{thm}

\begin{rem}
$\ $\begin{enumerate}[font=\normalfont] 

\item A representation $\rho:G\to{\rm PSL}(2,\R)$ is called \textit{Zariski
dense} if it is irreducible and non-parabolic, where non-papabolic
means $\rho(G)$ has no global fixed point on the boundary of $\H$.
It is clear that Fuchsian representations satisfying the extended
Schottky condition (see Section \ref{sec:Preliminaries-on-Extended})
are Zariski dense. 

\item Kim's result is way more general than the version stated above.
However, this version is sufficient for us. Also, the stated version
should be known before Kim; however, we cannot find a proper reference
earlier than this one. 

\end{enumerate}
\end{rem}

\subsubsection{Boundary-Preserving Isomorphic Representations}

\begin{defn}
\label{def:type-preserving} Let $\rho_{1},\rho_{2}$ be two geometrically
finite Fuchsian representations from $G(=\pi_{1}S)$ into ${\rm PSL}(2,\R)$.
We say $\rho_{1}$ and $\rho_{2}$ are \textit{boundary-preserving
isomorphic} if there exists an isomorphism $\iota:\rho_{1}(G)\to\rho_{2}(G)$
such that \begin{enumerate}[font=\normalfont]

\item $\iota$ is \textit{type-preserving}, i.e., $\iota$ sends
hyperbolic elements to hyperbolic elements and parabolic elements
to parabolic elements, 

\item $\iota$ is \textit{peripheral-structure-preserving}, i.e.,
$\g\in\rho_{1}(G)$ corresponds to a geodesic boundary of $S_{1}$
if and only $\iota(\g)\in\rho_{2}(G)$ corresponds to a geodesic boundary
of $S_{2}$.

\end{enumerate}
\end{defn}

\begin{thm}
[Fenchel-Nielsen Isomorphism Theorem, cf. Theorem 5.4 \cite{Kapovich:2009fk},  Theorem V.H.1 \cite{Maskit:1988th}]\label{thm: bilipsthitz }
Let $\rho_{1},\rho_{2}$ be two geometrically finite Fuchsian representations
and $S_{1}=\rho_{1}(G)\backslash\H$, $S_{2}=\rho_{2}(G)\backslash\H$.
Suppose there is a boundary-preserving isomorphism $\iota:\rho_{1}(G)\to\rho_{2}(G)$.
Then there exists an $\iota$-equivariant bilipschitz homeomorphism
$f:S_{1}\to S_{2}$. 
\end{thm}

We then lift $f$ to their universal coverings, and, thus, derive
an $\iota$-equivariant bilipschitz homeomorphism between universal
coverings (both are $\H$). By abusing the notation, we still denote
this homeomorphism by $f:\H\to\H$. More precisely, there exists a
constant $C>0$ such that for $x,y\in\H$ 
\[
\frac{1}{C}d(x,y)\leq d(f(x),f(x))\leq Cd(x,y).
\]

\begin{rem}
$\ $\begin{enumerate}[font=\normalfont] \label{rem:bilipschitz}

\item In Theorem 5.4 \cite{Kapovich:2009fk}, the $\iota$-equivariant
homeomorphism $f:S_{1}\to S_{2}$ is stated to be quasiconformal.
Nevertheless, it is well-known (cf. Mori's theorem) that quasiconformal
homeomorphisms are bilipschitz maps. 

\item \label{rem:-Tukia's-theorem}Tukia's isomorphism Theorem (cf.
Theorem 3.3 \cite{Tukia:1985vi}) points out that the boundaries of
these two Fuchsian groups are also strongly related. More precisely,
there exists an $\iota$-equivariant Hölder continuous homeomorphism
$q:\Lambda(\G_{1})\to\Lambda(\G_{2})$.

\end{enumerate}
\end{rem}

%% file: 3.Extended_Schottky.tex
\section{Extended Schottky Surfaces \label{sec:Preliminaries-on-Extended}}

In this section, following the notations in Dal'Bo-Peigné, we will
mostly use the Poincaré disk model $\D$. Nevertheless, one can easily
convert it to the upper-half plane model $\H$. Let us fix two integers
$N_{1}$, $N_{2}$ such that $N_{1}+N_{2}\geq2$ and $N_{2}\geq1$
and consider $N_{1}$ hyperbolic isometries $h_{1},...,h_{N_{1}}$
and $N_{2}$ parabolic isometries $p_{1},...,p_{N_{2}}$ satisfying
the following conditions:

\begin{enumerate}[font=\normalfont]

\item[(C1)] For $1\leq i\leq N_{1}$ there exists in $\vbdy\D=S^{1}$
a compact neighborhood $C_{h_{i}}$ of the attracting fixed point
$h_{i}^{+}$ of $h_{i}$ and a compact neighborhood $C_{h_{i}^{-1}}$
of the repelling fixed point $h_{i}^{-}$ of $h_{i}$ such that 
\[
h_{i}(S^{1}\backslash C_{h_{i}^{-1}})\subset C_{h_{i}}.
\]

\item[(C2)] For $1\leq i\leq N_{2}$ there exists in $S^{1}$ a compact
neighborhood $C_{p_{i}}$ of the unique fixed point $p_{i}^{\pm}$
of $p_{i}$ such that for all $n\in\Z^{*}:=\Z\backslash\{o\}$ 
\[
p_{i}^{n}(S^{1}\backslash C_{p_{i}})\subset C_{p_{i}}.
\]

\item[(C3)] The $2N_{1}+N_{2}$ neighborhoods introduced in ${\rm (C1)}$
and ${\rm (C2)}$ are pairwise disjoint.

\end{enumerate}

The group $\G=\langle h_{1},...,h_{N_{1}},p_{1},...,p_{N_{2}}\rangle\leq{\rm Isom}(\D)\cong{\rm PSL}(2,\R)$
is proved (cf. \cite{Dalbo:1996wh}) to be a non-elementary free group
which acts properly discontinuously and freely on $\D$.

\begin{defn}
\label{def:extended Schottky} We call $\G=\langle h_{1},...,h_{N_{1}},p_{1},...,p_{N_{2}}\rangle$
an \textit{extended Schottky group} if it satisfies conditions $({\rm C1}),\ ({\rm C2}),\ ({\rm C3})$,
and $N_{1}+N_{2}\geq3$. Moreover, if $\G$ is an extended Schottky
group and $R$ is the hyperbolic surface $\G\backslash\D$, then we
say that the corresponding Fuchsian representation $\rho$ (i.e.,
$\rho:\pi_{1}R\to{\rm PSL}(2,\R)$ such that $\rho(\pi_{1}R)=\G$)
satisfies\textit{ the extended Schottky condition}.
\end{defn}

\begin{center}\includegraphics[scale=0.35]{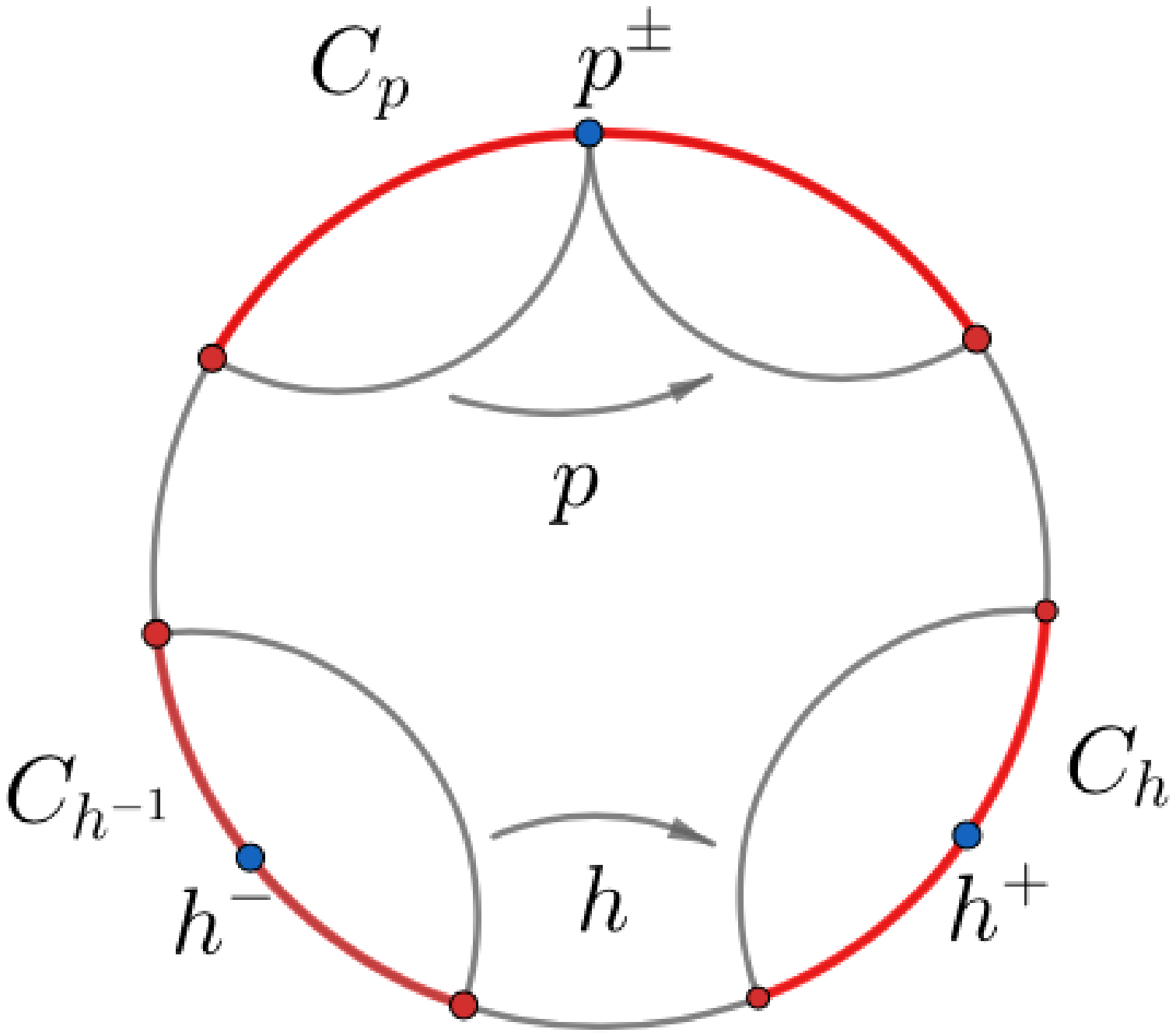}\end{center}

\begin{rem}
$\ $\begin{enumerate}[font=\normalfont]

\item If $N_{2}=0$ the groups $\G$ is a (classical) Schottky group.

\item Hyperbolic surface satisfying ${\rm (C1),\ (C2),\ (C3)}$ are
geometrically finite with infinite volume. 

\item For a hyperbolic surface satisfying ${\rm (C1),\ (C2),\ (C3)}$,
by the computation in Lemma \ref{lem:compute critical exponent},
one has the elementary parabolic groups $\langle p_{i}\rangle$ for
$1\leq i\leq N_{2}$ are of divergent type.

\item The definition of extended Schottky condition here (for hyperbolic
surfaces) is extracted from a more general definition for manifolds
with pinched negative curvatures (cf. \cite{Dalbo:1996wh,Dalbo:1998io}).

\end{enumerate}
\end{rem}

Let ${\cal A}^{\pm}=\left\{ h_{1}^{\pm1},...,h_{N_{1}}^{\pm1},p_{1},...,p_{N_{2}}\right\} $.
For $a\in{\cal A}^{\pm}$ denote by $U_{a}$ the convex hull in $\D\cup\vbdy\D$
of the set $C_{a}$. For extended Schottky surfaces, we have the following
important and very useful lemma.

\begin{lem}
\label{lem:d(x,y)>d(x,o)+d(y,0)+C}Let $\G$ be an extended Schottky
group. Fix $o\in\D$, then there exists an universal constant $C>0$
(depending only on generators of $\G$ and the fixed point $o$) such
that for every $a_{1},a_{2}\in{\cal A}^{\pm}$ satisfying $a_{1}\neq a_{2}^{\pm1}$,
and for every $x\in U_{a_{1}}$ and $y\in U_{a_{2}}$, one has 
\[
d(x,y)\geq d(x,o)+d(y,o)-C.
\]
 
\end{lem}

\begin{rem}
The above lemma is well-known. The version that we stated is taken
from Lemma 4.4 \cite{Iommi:2016tv}. 
\end{rem}

\subsection{Coding of Closed Geodesics }

In this subsection, we plan to present a coding of closed geodesics
on extended Schottky surfaces. This symbolic coding is given in Dal'Bo-Peigné
\cite{Dalbo:1996wh} (the case of ${\cal P}=\emptyset$ in their notation). 

Throughout this subsection, let $S$ be a surface with negative Euler
characteristic and $\rho_{1}$, $\rho_{2}$ be two boundary-preserving
isomorphic Fuchsian representations, from $G=\pi_{1}S$ into ${\rm PSL}(2,\R)$,
satisfying the extended Schottky condition. For $i=1$,2, we denote
$\G_{i}=\rho_{i}(G)$, $S_{i}=\G_{i}\backslash\D$, and $\Lambda(\G_{i})$
denotes the limit set of $\G_{i}$. 

Since $\rho_{1}$ and $\rho_{2}$ are boundary-preserving isomorphic
and satisfying the extended Schottky condition, we write $G=\langle h_{1},h_{2},...,h_{N_{1}},p_{1},p_{2},...,p_{N_{2}}\rangle$
where $h_{j}$ (resp. $p_{k}$) is called hyperbolic (resp. parabolic)
and it corresponds to a hyperbolic (resp. parabolic) element $\rho_{i}(h_{j})$
(resp. $\rho_{i}(p_{k}$)). We denote the set of generators by $\mathcal{A}=\{h_{1},h_{2},...,h_{N_{1}},p_{1},p_{2},...,p_{N_{2}}\}$. 

We first work on one fixed extended Schottky surface, say $S_{1}$.
In the following, we recall definitions and summarize several useful
propositions from Dal'Bo-Peigné \cite{Dalbo:1996wh} about the coding
of the geodesics on $S_{1}$.

\begin{defn}
$\ $\begin{enumerate}[font=\normalfont]\label{def:TMS_Coding}

\item Let $\mathcal{A}=\{h_{1},h_{2},...,h_{N_{1}},p_{1},p_{2},...,p_{N_{2}}\}$,
the countable Markov shift $(\TMS,\s)$ associated with $S_{1}$ is
defined as 
\[
\TMS=\{x=(a_{i}^{n_{i}})_{i\geq1}:\ a_{i}\in{\cal A},\ n_{i}\in\Z{\rm ^{*},\ and\ }a_{i}\neq a_{i+1}^{\pm}\}\ {\rm where}\ \Z^{*}=\Z\backslash\{0\},
\]
and the shift map $\sigma(a_{1}^{n_{1}}a_{2}^{n_{2}}a_{3}^{n_{3}}...)=a_{2}^{n_{2}}a_{3}^{n_{3}}...$.;

\item $\Lambda_{1}^{0}$ is a subset of $\L(\G_{1})$ defined as
\[
\L_{1}^{0}=\L(\G_{1})\backslash\{\G_{1}\xi:\ \mbox{\ensuremath{\xi}{\rm \ is a fixed point of \ensuremath{\rho_{1}(\alpha)} for} }\alpha\in{\cal A}\};{\rm \ and}
\]

\item ${\cal G}_{S_{1}}$ is the set of all closed geodesics on $S_{1}$
except those corresponding to hyperbolic elements in ${\cal A}$.

\end{enumerate}
\end{defn}

\begin{prop}
[Coding Property and the Geometric Potential ] \label{prop:coding_property}$\ $\begin{enumerate}[font=\normalfont]

\item ${\rm (p.759\ }$\cite{Dalbo:1996wh}$)$ There exists a bijection
$\omega_{1}:\L_{1}^{0}\to\TMS$.

\item ${\rm (p.760\ }$\cite{Dalbo:1996wh}$)$ The Bowen-Series
map $T:\L_{1}^{0}\to\L_{1}^{0}$ is given by $T(\xi)=\omega_{1}^{-1}(\s(\omega_{1}(\xi))$
for $\xi\in\Lambda_{1}^{0}$. 

\item ${\rm (Lemma}\ {\rm II.1\ }$\cite{Dalbo:1996wh}${\rm )}$
There exists a bijection (up to cyclic permutations) ${\cal H}:{\cal G}_{S_{1}}\to{\rm Fix}(\TMS)$
where ${\rm Fix}(\TMS)=\cup_{n}{\rm Fix}^{n}(\TMS)$ is the set of
fixed points of $\sigma$.

\item ${\rm (p.759\ }$\cite{Dalbo:1996wh}$)$ Let $\tau:\TMS\to\R$
be the geometric potential (relative to $T$), that is, 
\[
\tau(x):=-\log|T'(\omega_{1}^{-1}(x))|=B_{\omega_{1}^{-1}(x)}(o,\rho_{1}(a_{1}^{n_{1}})o),{\rm \ }where\ x=a_{1}^{n_{1}}a_{2}^{n_{2}}...\in\TMS.
\]
Suppose $\g\in\G_{1}$ is a hyperbolic element and $\w_{1}(\g^{+})=\overline{a_{1}^{n_{1}}...a_{k}^{n_{k}}}\in{\rm Fix}^{k}(\TMS)$,
then 
\[
l_{1}[\g]=S_{k}(\tau(\w_{1}(\g^{+})).
\]

\item ${\rm (Lemma}\ {\rm II.4\ }$\cite{Dalbo:1996wh}$)$ There
exist $K,C>0$ such that ${\displaystyle S_{n}\tau(x)\geq C}$ for
all $n>K$ and $x\in\TMS$.

\item ${\rm (Lemma}\ {\rm V.2,V.5\ }$\cite{Dalbo:1996wh}$)$ \label{lem:locally holder}
$\tau$ is locally Hölder continuous. 

\end{enumerate}
\end{prop}

Furthermore, the countable Markov shift $(\TMS,\s)$ derived above
satisfies the following two favorable conditions.

\begin{prop}
[Properties of the Markov Shift]\label{thm:coding of geodesic flow}
Let $(\TMS,$$\s$) be the countable Markov shift associated to $S_{1}$.
Then

\begin{enumerate}[font=\normalfont]

\item The Markov shift $(\TMS,\sigma)$ satisfies the BIP property;
and

\item If $N_{1}+N_{2}\geq3$, then $(\TMS,\sigma)$ is topologically
mixing.

\end{enumerate}
\end{prop}

\begin{proof}
Taking the finite set to be ${\cal A}=\{h_{1},h_{2},...,h_{N_{1}},p_{1},p_{2},...,p_{N_{2}}\}$,
then it is clear that $(\TMS,\sigma)$ satisfies the BIP property
(see Definition \ref{def:BIP}). The topologically mixing property
for Markov shifts is a combinatorics condition: 

\textbf{Claim:} For every $x,y\in\{a_{i}^{m}:\ a_{i}\in{\cal A},\ m\in\Z\}$,
there exists $N=N(x,y)\in\N$ such that for all $k>N$ there is an
admissible word of length $k$ of the form $xa_{2}^{n_{2}}a_{3}^{n_{3}}...a_{k-1}^{n_{k-1}}y$
for some $n_{i}\in\Z^{*}$ and $i=2,...,k-1$.

pf. Recall that $\TMS=\{x=(a_{i}^{n_{i}})_{i\geq1}:\ a_{i}\in{\cal A},\ n_{i}\in\Z{\rm ^{*},\ and\ }a_{i}\neq a_{i+1}^{\pm}\}$.
Since $N_{1}+N_{2}\geq3,$ we have at least three distinct elements
in ${\cal A}$, say $a_{1},a_{2},a_{3}$. Pick two elements $x,y$
in $\{a_{i}^{m}:\ a_{i}\in{\cal A},\ m\in\Z\}$, w.l.o.g., say $x=a_{1}^{m_{1}}$
and $y=a_{2}^{m_{2}}$. 

For $k=2t+2$ for any $t\in\N$, then the following word is admissible:
\[
a_{1}^{m_{1}}\underset{t\ {\rm paris}}{\underbrace{(a_{2}a_{3})...(a_{2}a_{3})}}a_{2}^{m_{2}}.
\]

For $k=2t+3$ for any $t\in\N$, then the following word is admissible:
\[
a_{1}^{m_{1}}\underset{t\ {\rm paris}}{\underbrace{(a_{2}a_{3})...(a_{2}a_{3})}}a_{1}a_{2}^{m_{2}}.
\]

\end{proof}

Using a standard argument in symbolic dynamics, we observe the following
handy lemma for the geometric potential $\tau$.

\begin{lem}
\label{lem:roof fcn >0}There exists a locally Hölder continuous functions
$\tau'$ such that $\tau\sim\tau'$ and $\tau'$ is bounded away from
zero.
\end{lem}

\begin{proof}
By the above proposition, we know there exist $K,C>0$ such that ${\displaystyle \tau+\tau\circ\sigma+...+\tau\circ\sigma^{m}\geq C}$
for all $m>K$. Let $\lambda=\frac{1}{K}$ and consider $h'(x)={\displaystyle \sum_{n=0}^{K-1}}a_{n}\cdot\tau\circ\sigma^{n}(x)$
where $a_{n}=1-n\lambda$. Notice that $a_{0}=1$, $a_{K-1}=\lambda$
and $a_{K}=0$. Moreover, we have $a_{n}-a_{n-1}=-\lambda$ for $n=1,2,...,K$.

Therefore, 
\begin{alignat*}{1}
h'(x)-h(\sigma x) & ={\displaystyle \sum_{n=0}^{K-1}}a_{n}\cdot\tau\circ\sigma^{n}(x)-{\displaystyle \sum_{n=0}^{K-1}}a_{n}\cdot\tau\circ\sigma^{n+1}(x)\\
 & =a_{0}\cdot\tau(x)-\lambda\cdot(\tau\circ\sigma x+\tau\circ\sigma^{2}x+...+\tau\circ\sigma^{K-1}x)-a_{K-1}\tau\circ\sigma^{K}(x)\\
 & =\tau(x)-\lambda\sum_{n=1}^{K}\tau\circ\sigma^{n}x.
\end{alignat*}

Let $\tau'(x):={\displaystyle \lambda\sum_{n=1}^{K}\tau\circ\sigma^{n}x}$.
It is clear that $\tau'(x)$ is locally Hölder; moreover, we have
\[
\tau'(x)=\lambda\sum_{n=1}^{K}\tau\circ\sigma^{n}x\geq\frac{C}{K}>0.
\]

\end{proof}

Notice that the coding above is completely determined by the type
of generators (i.e., hyperbolic or parabolic) in $\G_{1}$. Because
$\G_{1}$ and $\G_{2}$ are boundary-preserving isomorphic, repeating
the same construction as above for $\G_{2}$, for $S_{2}$ we derive
the same countable Markov shift $(\TMS,\s)$ as for $S_{1}$. In other
words, the same Proposition \ref{prop:coding_property} holds for
$S_{2}$. More precisely, there exists a bijection $\omega_{2}:\L_{2}^{0}\to\TMS$
and the geometric potential $\kappa:\TMS\to\R$ given by $\kappa(x):=B_{\omega_{2}^{-1}(x)}(o,\rho_{2}(a_{1}^{n_{1}})o)$
for $x=a_{1}^{n_{1}}a_{2}^{n_{2}}...\in\TMS$. Furthermore, $\kappa$
is cohomologus to a locally Hölder continuous function $\kappa'$
which is bounded away from zero (i.e., Lemma \ref{lem:roof fcn >0}).

\begin{rem}
$\ $\begin{enumerate}[font=\normalfont]

\item\label{rem:coding_commute} Suppose $\iota:\G_{1}\to\G_{2}$
is a type-preserving isomorphism. Then by Tukia's isomorphism theorem
(cf. Remark \ref{rem:bilipschitz}.\ref{rem:-Tukia's-theorem}) there
exists an $\iota-$equivariant homeomorphism $q:\L(\G_{1})\to\L(\G_{2})$.
One can also prove that for $\xi\in\L_{1}^{0}$ we have $\w_{2}(\xi)=\w_{1}(q(\xi))$.
Moreover, we can write $\kappa(x)=B_{(\w_{1}\circ q)^{-1}(x)}(o,(\iota\circ\rho_{1})(a_{1}^{n_{1}})\cdot o)$
where $a_{1}^{n_{1}}$ is the first element of $\w_{1}^{-1}(x)$. 

\item Noticing that since $\tau$ and $\tau'$ (constructed in Corollary
\ref{lem:roof fcn >0}) are cohomologous, the thermodynamics for $\tau$
(resp. $\kappa$) and $\tau'$ (resp. $\kappa'$) are the same. Therefore,
for brevity, we will abuse our notation and continue to denote the
function $\tau'$ by $\tau$ and, similarly, $\kappa'$ by $\kappa$. 

\end{enumerate}
\end{rem}

\subsection{Phase Transition of the Geodesic Flow}

We continue this subsection with the same notations and assumptions
as the previous subsection. Recall $D=\{(x,y)\in\R^{2}:\ x\geq0,y\geq0\}\backslash(0,0)$.
Throughout, let $\rho_{1}$ and $\rho_{2}$ be two boundary-preserving
isomorphic Fuchsian representations satisfying the extended Schottky
condition.

\begin{lem}
\label{lem:compute critical exponent}Suppose $(a,b)\in D$. For any
parabolic element $p\in G$ (i.e., $\rho_{1}(p)$ and $\rho_{2}(p)$
are parabolic), we have $\delta_{\langle p\rangle}^{a,b}=\inf\left\{ t\in\R:\ Q_{\langle p\rangle}^{a,b}(t)<\infty\right\} =\frac{1}{2(a+b)}$
where $Q_{\langle p\rangle}^{a,b}(t)=\sum_{n\in\Z}e^{-t(d^{a,b}(o,p^{n}))}$.
For $h\in\G$ is hyperbolic (i.e., $\rho_{1}(h)$ and $\rho_{2}(h)$
are hyperbolic), then $\delta_{\langle h\rangle}^{a,b}=0$.
\end{lem}

\begin{proof}
Let $p\in G$ be a parabolic element. Without loss generality, we
can assume $\rho_{i}(p):\H\to\H$ to be the Möbius transformation
$\rho_{i}(p)(z)=z+c_{i}$ for $i=1,2$ where $c_{i}\in\R$. Then direct
computation shows that 
\[
d(i,\rho_{i}(p^{n})(i))=d(i,i+nc_{i})=\log\frac{\sqrt{(nc_{i})^{2}+4}+|nc_{i}|}{\sqrt{(nc_{i})^{2}+4}-|nc_{i}|}.
\]

Notice that 
\[
\frac{\sqrt{(nc_{i})^{2}+4}+|nc_{i}|}{\sqrt{(nc_{i})^{2}+4}-|nc_{i}|}=\frac{2n^{2}c_{i}^{2}+4+2|nc_{i}|\sqrt{(nc_{i})^{2}+4}}{4},
\]

so when $|n|$ is big enough (say $|n|>M_{p})$, there exist $m_{i}$
and $M_{i}$ such that 
\[
2\log|n|+m_{i}\leq d(i,i+nc_{i})\leq2\log|n|+M_{i}.
\]
 Converting the above inequalities to the disk model, we have 

\[
2\log|n|+m_{i}\leq d(o,p^{n}o)\leq2\log|n|+M_{i}.
\]

Therefore,
\begin{alignat*}{1}
Q_{\langle p\rangle}^{a,b}(t)= & \sum_{n\in\Z}e^{-t\cdot d^{a,b}(o,p^{n}o)}\\
= & \sum_{|n|\leq M_{p}}e^{-t\cdot d^{a,b}(o,p^{n}o)}+\sum_{|n|>M_{p}}e^{-t\cdot d^{a,b}(o,p^{n}o)},
\end{alignat*}
where ${\displaystyle \sum_{|n|\leq M_{p}}e^{-t\cdot d^{a,b}(o,p^{n}o)}}<\infty$
is a finite sum. Furthermore, for $|n|>M$ one has

\begin{alignat*}{1}
-tad(o,\rho_{1}(p^{n})(o))-tbd(o,\rho_{1}(p^{n})(o)) & \geq-ta(2\log|n|+M_{1})-tb(2\log|n|+M_{2})\\
 & =-t\underset{C_{1}^{a,b}(p)}{\underbrace{(aM_{1}+bM_{2})}}-2t(a+b)\log|n|
\end{alignat*}
and 

\begin{alignat*}{1}
-tad(o,\rho_{1}(p^{n})(o))-tbd(o,\rho_{1}(p^{n})(o)) & \leq-ta(2\log|n|+m_{1})-tb(2\log|n|+m_{2})\\
 & =-t\underset{C_{2}^{a,b}(p)}{\underbrace{(am_{1}+bm_{2})}}-2t(a+b)\log|n|.
\end{alignat*}
Hence 
\[
(\frac{1}{C_{1}^{a,b}(p)})^{t}\sum_{|n|>M_{p}}(\frac{1}{|n|})^{2t(a+b)}\leq\sum_{|n|>M_{p}}e^{-t\cdot d^{a,b}(o,p^{n}o)}\leq(\frac{1}{C_{2}^{a,b}(p)})^{t}\sum_{|n|>M_{p}}(\frac{1}{|n|})^{2t(a+b)},
\]
 and, thus, $\delta_{\langle p\rangle}^{a,b}=\frac{1}{2(a+b)}$. 

For each hyperbolic element $h\in G$, and 
\begin{alignat*}{1}
Q_{\langle h\rangle}^{a,b}(t)= & \sum_{n\in\Z}e^{-t\cdot d^{a,b}(o,h^{n}o)}\\
= & \sum_{n\in\Z}e^{-tad(o,\rho_{1}(h^{n})o)-tbd(o,\rho_{_{2}}(h^{n})o)}\\
= & 2\sum_{n\in\N}e^{-tanB_{\rho_{1}(h)^{+}}(o,\rho_{1}(h)o)-tnbB_{\rho_{2}(h)^{+}}(o,\rho_{2}(h)o)}\\
= & 2\sum_{n\in\N}e^{-tn(aB_{\rho_{1}(h)^{+}}(o,\rho_{1}(h)o)+bB_{\rho_{2}(h)^{+}}(o,\rho_{2}(h)o))}.
\end{alignat*}
Since $B_{\rho_{i}(h){}^{+}}(o,\rho_{i}(h)o)>0$ for $i=1,2$, we
have $\delta_{\langle h\rangle}^{a,b}=0.$
\end{proof}

Recall that the Markov shift $(\Sigma^{+},\sigma)$ defined above
(see Definition \ref{def:TMS_Coding}) for $\rho_{1},\rho_{2}$ is
topologically mixing and satisfying the BIP property. Also, the geometric
potentials $\tau$,$\k$ defined above (see Proposition \ref{prop:coding_property})
are locally Hölder and bounded away from zero. Therefore, we are in
the scenario that was introduced in Section \ref{sec:Preliminaries-on-Thermodynamic}.

\begin{lem}
\label{lem:phas transition and critical exponent} Let $\rho_{1}$
and $\rho_{2}$ be two boundary-preserving isomorphic Fuchsian representations
satisfying the extended Schottky condition. Let $(\Sigma^{+},\sigma)$
be Markov shift and $\tau$,$\k$ be the geometric potentials defined
in the above subsection. 

Then for $a,b\geq0,$ 
\[
P_{\sigma}(-t(a\tau+b\k))=\begin{cases}
{\rm infinite,} & {\rm for\ }t<\delta_{\langle p\rangle}^{a,b};\\
{\rm analytic,} & {\rm {\rm for\ }}t>\delta_{\langle p\rangle}^{a,b}.
\end{cases}
\]

\end{lem}

\begin{proof}
By definition, we have 
\begin{alignat*}{1}
P_{\s}(-t(a\tau+b\k)) & =\lim_{n\to\infty}\frac{1}{n+1}\log\left(\sum_{\overset{x\in{\rm Fix}^{n}}{\underset{}{}}}\exp(-t(aS_{n}\tau+bS_{n}\k))\cdot\chi_{[h_{1}]}\right)\\
 & =\lim_{n\to\infty}\frac{1}{n+1}\log\left(\sum_{x=\overline{h_{1}x_{2}...x_{n+1}}}\exp(-t(aS_{n}\tau+bS_{n}\k))\right)
\end{alignat*}

Notice that 
\[
{\rm Fix}^{n+1}(\Sigma^{+})=\left\{ \overline{a_{1}^{m_{1}}a_{2}^{m_{2}}....a_{n+1}^{m_{1}}}:\ a_{i}\in{\cal A},\ a_{i}\neq a_{i+1}^{\pm1},\ {\rm and}\ m_{i}\in\Z^{*}\ {\rm for\ }i=1,2,..,n+1\right\} .
\]

For each $k\in\N$ and set $n+1=k(N_{1}+N_{2}-1)$, let's consider
a subset $B^{k}\subset{\rm Fix}^{n+1}$ defined as 

\[
B^{k}=\left\{ \overline{h_{1}a_{1}^{m_{1}}...a_{n}^{m_{n}}}\in{\rm Fix}^{n+1}:\ a_{i+j(N_{1}+N_{2}-1)}=\begin{cases}
h_{i+1}, & 1\leq i\leq N_{1}-1\\
p_{i+1-N_{1}}, & N_{1}\leq i\leq N_{1}+N_{2}-1
\end{cases}\right\} .
\]
 In other words, elements $b\in B^{k}$ are in the following form:

\[
b=\overline{h_{1}\underbrace{h_{2}^{m_{1}}...h_{N_{1}}^{m_{N_{1-1}}}p_{1}^{m_{N_{1}}}...p_{N_{2}}^{m_{N_{1}+N_{2}-1}}}\ ...\ \underbrace{h_{2}^{m_{(k-1)(N_{1}+N_{2}-1)}}...p_{N_{2}}^{m_{k(N_{1}+N_{2}-1)}}}.}
\]

For brevity, let's denote $N_{1}+N_{2}-1$ by $N_{3}$, then we have
for $\xi_{0}\in\Lambda_{1}^{0}$

\begin{alignat*}{1}
P_{\s}(-t(a\tau+b\k)) & \geq\lim_{k\to\infty}\frac{1}{kN_{3}}\log\left(\sum_{\overset{\xi=\rho_{1}(x)\xi_{0}}{\underset{x\in B^{k}}{}}}\exp(-t(aS_{kN_{3}}\tau+bS_{kN_{3}}\k))\right)\\
 & =\lim_{k\to\infty}\frac{1}{kN_{3}}\log\left(\sum_{\overset{\xi=\rho_{1}(x)\xi_{0}}{\underset{x\in B^{k}}{}}}\exp(f(a,b,t,kN_{3}))\right)
\end{alignat*}
where 
\[
f(a,b,t,n)=-t(\sum_{i=1}^{n}aB_{\w_{1}^{-1}(\s^{i}x)}(o,\rho_{1}(x_{i+1})o)+bB_{\w_{2}^{-1}(\s^{i}x)}(o,\rho_{2}(x_{i+1})o)).
\]
Because $B_{\xi}(x,y)\leq d(x,y)$ we have, 
\begin{alignat*}{1}
P_{\s}(-t(a\tau+b\k)) & \geq\lim_{k\to\infty}\frac{1}{kN_{3}}\log\sum_{\underset{x\in B^{k}}{\xi=\rho_{1}(x)\xi_{0}}}\exp\left(-t(\sum_{i=1}^{kN_{3}}ad(o,\rho_{1}(x_{i+1})o)+bd(o,\rho_{2}(x_{i+1})o)\right)\\
 & =\lim_{k\to\infty}\frac{1}{kN_{3}}\log\left(\sum_{\overset{\xi=\rho_{1}(x)\xi_{0}}{\underset{x\in B^{k}}{}}}\exp\left(-t\sum_{i=1}^{kN_{3}}d^{a,b}(o,x_{i+1}o)\right)\right)
\end{alignat*}
Moreover, by the definition of $B^{k}$ one has 

\begin{alignat*}{1}
\sum_{\underset{x\in B^{k}}{\xi=\rho_{1}(x)\xi_{0}}}\exp\left(-t\sum_{i=1}^{kN_{3}}d^{a,b}(o,x_{i+1}o)\right) & =\\
e^{-td^{a.b}(o,h_{1}o)\cdot}\sum_{(m_{1},...,m_{kN_{3}})\in(\Z^{*}){}^{kN_{3}}} & \exp\left(-t\sum_{i=1}^{kN_{3}}d^{a,b}(o,a_{i}^{m_{i}}o)\right).
\end{alignat*}

Also, notice that 
\begin{alignat*}{1}
\sum_{(m_{1},...,m_{kN_{3}})\in(\Z^{*})^{kN_{3}}} & \exp\left(-t\sum_{i=1}^{kN_{3}}d^{a,b}(o,a_{i}^{m_{i}}o)\right)\\
=\prod_{i=1}^{kN_{3}} & \sum_{m_{i}\in\Z^{*}}\exp\left(-t\sum_{i=1}^{kN_{3}}d^{a,b}(o,a_{i}^{m_{i}}o)\right)\\
= & \left(\prod_{i=2}^{N_{1}}\sum_{m\in\Z^{*}}e^{-td^{ab}(o,h_{i}^{m}o)}\right)^{k}\left(\prod_{i=1}^{N_{2}}\sum_{m\in\Z^{*}}e^{-td^{ab}(o,p_{i}^{m}o)}\right)^{k}.
\end{alignat*}

Hence, 
\begin{alignat*}{1}
P_{\s}(-t(a\tau+b\k)) & \geq\\
\lim_{k\to\infty}\frac{1}{kN_{3}}\log & \left(e^{-td^{a.b}(o,h_{1}o)}\left(\prod_{i=2}^{N_{1}}\sum_{m\in\Z^{*}}e^{-td^{ab}(o,h_{i}^{m}o)}\right)^{k}\left(\prod_{i=1}^{N_{2}}\sum_{m\in\Z^{*}}e^{-td^{ab}(o,p_{i}^{m}o)}\right)^{k}\right)\\
 & =\frac{1}{N_{3}}\left(\log\left(\prod_{i=2}^{N_{1}}\sum_{m\in\Z^{*}}e^{-td^{ab}(o,h_{i}^{m}o)}\right)\left(\prod_{i=1}^{N_{2}}\sum_{m\in\Z^{*}}e^{-td^{ab}(o,p_{i}^{m}o)}\right)\right)\\
 & =\frac{1}{N_{3}}\log\left(\prod_{g\in{\cal A\backslash}h_{1}}\left(Q_{\langle g\rangle}^{a,b}(t)-1\right)\right)
\end{alignat*}
where $Q_{\langle g\rangle}^{a,b}(t)={\displaystyle \sum_{m\in\Z}e^{-td^{ab}(o,g^{m}o)}=1+\sum_{m\in\Z^{*}}e^{-td^{ab}(o,g^{m}o)}}$.

In the following, we derive an upper bound for $P_{\s}(-t(a\tau+b\k))$.
Let $(\xi_{t}^{i})$ be the end of the geodesic ray $[o,\omega_{1}^{-1}(\s^{i+1}x))$.
Then by Lemma \ref{lem:d(x,y)>d(x,o)+d(y,0)+C}, we have 
\begin{alignat*}{1}
\tau(\sigma^{i}x) & =B_{\w_{1}^{-1}(\sigma^{i}x)}(o,\rho_{1}(x_{i})o)\\
 & =B_{\w_{1}^{-1}(\sigma^{i+1}x)}(\rho_{1}^{-1}(x_{i})o,o)\\
 & =\lim_{t\to\infty}d(\xi_{t}^{i},\rho_{1}(x_{i})o)-d(\xi_{t}^{i},o)\\
 & \geq\left(d(\xi_{t}^{i},o)-d(o,\rho_{1}(x_{i})o)-C_{1}\right)-d(\mbox{\ensuremath{\xi}}_{t}^{i},o)\\
 & =d(o,\rho_{1}(x_{i})o)-C_{1}
\end{alignat*}
 Similarly, we have $\k(\s^{i}x)\geq d(o,\rho_{2}(x_{i})o)-C_{2}$
for some constant $C_{2}$. 

Thus, 
\[
e^{-t(a\tau(\s^{i}x)+a\k(\s^{i}x))}\leq e^{t(aC_{1}+bC_{2})}e^{-t(d^{a,b}(o,x_{i}o))}.
\]
Hence, 

\begin{alignat*}{1}
P_{\s}(-ta\t-tb\k) & \leq\lim_{n\to\infty}\frac{1}{n}\log\left(\sum_{a_{1},...,a_{n}}\sum_{m_{1},...,m_{n}\in\Z^{*}}\prod_{i=1}^{n}e^{t(aC_{1}+bC_{2})}e^{-t(d^{a,b}(o,a_{i}^{m_{i}}o))}\right)\\
 & =t(aC_{1}+bC_{2})+\log\left(\prod_{g\in{\cal A}}\left(Q_{\langle g\rangle}^{a,b}(t)-1\right)\right).
\end{alignat*}

Then, by Lemma \ref{lem:compute critical exponent} we have 
\[
P_{\sigma}(-t(a\tau+b\k))=\begin{cases}
{\rm infinite,} & {\rm for\ }t<\delta_{\langle p\rangle}^{a,b};\\
{\rm finite,} & {\rm {\rm for\ }}t>\delta_{\langle p\rangle}^{a,b}.
\end{cases}
\]

Finally, by Theorem \ref{thm:Analyticty pressure}, we know the finiteness
of the pressure function implies the analyticity. 
\end{proof}

\begin{rem}
When $a$ (or $b)$ is zero, we are back to the well-know result:
\[
P_{\sigma}(-t\t)=\begin{cases}
\infty, & t\geq\frac{1}{2}\\
{\rm finite}, & t<\frac{1}{2}.
\end{cases}
\]
 
\end{rem}

\begin{lem}
\label{lem:P(-at-bk)=00003D0}For each $(a,b)\in D$ there exists
a unique $t_{a,b}\in(\frac{1}{2(a+b)},\infty)$ such that 
\[
P_{\sigma}(-t_{a,b}(a\t+b\k))=0.
\]

\end{lem}

\begin{proof}
Let $(a,b)$ be a point in $D$ and $f(t)=P_{\s}(-t(a\tau+b\k))$.
It is obvious that $-t(a\tau+b\k)$ is a locally Hölder continuous
function. By Theorem \ref{thm:Analyticty pressure}, $f(t)$ is real
analytic on $t$ when $P_{\s}(-t(a\tau+b\k))<\infty$. Let $K=\{t\in\R:\ f(t)<\infty\}$.
Then for $t_{0}\in K$ one has 
\[
\left.\frac{d}{dt}f(t)\right|_{t=t_{0}}=-\int(a\tau+b\k)\dd\mu_{-t_{0}(a\tau+b\k)}<-(ac+bc)<0
\]
 where $\t,\k>c>0$ and $\mu_{-t_{0}(a\tau+b\k)}$ is the equilibrium
state of $-t_{0}(a\tau+b\k)$.

Hence, $f(t)=P_{\s}(-t(a\tau+b\k))$ is real analytic and strictly
decreasing on $K$. Moreover, we know $P_{\sigma}(-t(a\tau+b\k))<0$
when $t$ and is positive and big enough. More precisely, because
$\kappa>c>0$, we know $P_{\sigma}(-t(a\tau+b\kappa))<P_{\sigma}(-ta\tau)-tbc$.
Furthermore, we know that $P_{\sigma}(-h_{top}(S_{1})\tau)=0$, so
when $ta>h_{top}(S_{1})$ we have $P_{\sigma}(-ta\tau)<0$. Therefore,
it remains to say there exists $t'_{a,b}\in(\frac{1}{2(a+b)},\infty)$
such that $0<P_{\sigma}(-t'_{a,b}(a\t+b\k))<\infty$. 

Notice that by the computation made in Lemma \ref{lem:compute critical exponent},
for a parabolic elements $p\in G$ and for $t>\frac{1}{2(a+b)}$,
\begin{alignat*}{1}
Q_{\langle p\rangle}^{a,b}(t)-1= & -1+\sum_{|n|\leq M_{p}}e^{-t\cdot d^{a,b}(o,p^{n}o)}+\sum_{|n|>M_{p}}e^{-t\cdot d^{a,b}(o,p^{n}o)}\\
> & (\frac{1}{C_{1}^{a,b}(p)})^{t}\sum_{|n|>M_{p}}(\frac{1}{|n|})^{2t(a+b)}\\
> & (\frac{1}{C_{1}^{a,b}(p)})^{t}\cdot2\int_{M_{p}+1}^{\infty}x^{-2t(a+b)}\dd x\\
= & (\frac{1}{C_{1}^{a,b}(p)})^{t}\cdot2\cdot\frac{1}{2t(a+b)-1}\cdot(\frac{1}{M_{p}+1})^{2t(a+b)-1}>0.
\end{alignat*}

Moreover, 
\begin{alignat*}{1}
\log\left(Q_{\langle p\rangle}^{a,b}(t)-1\right) & >-t\log(C_{1}^{a,b}(p))+\log2+\log(\frac{1}{2t(a+b)-1})+(2t(a+b)-1)\log(\frac{1}{M_{p}+1})\\
 & >0,\ {\rm when\ }t\ {\rm is\ {\rm big\ enough,}}
\end{alignat*}
 because $\log(\frac{1}{2t(a+b)-1})\to\infty\ {\rm as\ }t\to(\frac{1}{2(a+b)})^{+}$
and other terms remain bounded when $t\to(\frac{1}{2(a+b)})^{+}.$

For a hyperbolic elements $h\in G$, 
\[
Q_{\langle h\rangle}^{a,b}(t)-1=2\sum_{n\in\N}e^{-tn\cdot c_{a,b}(h)}=\frac{2}{e^{t\cdot c_{a,b}(h)}-1}
\]
where $c_{a,b}(h)=(aB_{\rho_{1}(h)^{+}}(o,\rho_{1}(h)o)+bB_{\rho_{2}(h)^{+}}(o,\rho_{2}(h)o))$,
one has 
\[
\log\left(Q_{\langle h\rangle}^{a,b}(t)-1\right)=\log2+\log(e^{t\cdot c_{a,b}(h)}-1)
\]
which remains bounded when $t\to(\frac{1}{2(a+b)})^{+}.$

By repeating the argument above for $g\in{\cal A\backslash}h_{1}$
and using the computation in Lemma \ref{lem:phas transition and critical exponent},
we can choose $t'_{a,b}\in(\frac{1}{2(a+b)},0)$ such that 
\[
\infty>P_{\s}(t'_{a,b}(a\t+b\k))>\frac{1}{N_{3}}\log\left(\prod_{g\in{\cal A\backslash}h_{1}}\left(Q_{\langle g\rangle}^{a,b}(t)-1\right)\right)>0.
\]

\end{proof}

\begin{thm}
\label{thm:P(-at-bk)=00003D0 analytic} The set $\{(a,b)\in D:\ P_{\sigma}(-a\tau-b\kappa)=0\}$
is a real analytic curve.
\end{thm}

\begin{proof}
By Lemma \ref{lem:P(-at-bk)=00003D0}, it makes sense to discuss solutions
to $P_{\sigma}(-a\tau-b\kappa)=0$. Moreover, for $(a,b)\in D$ such
that $f(a,b)=P_{\sigma}(-a\tau-b\k)<\infty$, we have $f(a,b)$ is
real analytic on both variables, and 
\[
\left.\partial_{b}f(a,b)\right|_{(a,b)=(a_{0},b_{0})}=-\int\kappa\dd\mu_{-a_{0}\t-b_{0}\k}<-c
\]
where $\tau,\k>c>0$ and $\mu_{-a_{0}\tau-b_{0}\k}$ is the equilibrium
state of $-a_{0}\t-b_{0}\k$. 

Therefore, by the Implicit Function Theorem we have the solutions
to $P_{\s}(-a\t-b\k)=0$ in $D$ is real analytic, i.e., one has $b=b(a)$
is real analytic on $a$.
\end{proof}

%% file: 4.Manhattan_Curve.tex
\section{The Manhattan Curve\label{sec:The-Manhattan-Curve}}

\subsection{The Manhattan Curve, Critical Exponent, and Gurevich Pressure}

For any pair of Fuchsian representations $\rho_{1},\rho_{2}$, we
recall that the Manhattan curve ${\cal C}(\rho_{1},\rho_{2})$ of
$\rho_{1}$ and $\rho_{2}$ is the boundary of the convex set 
\[
\{(a,b)\in\R^{2}:\ {\displaystyle Q_{\rho_{1},\rho_{2}}^{a,b}(s)\mbox{ has critical exponent }1\}}
\]
where $Q_{\rho_{1},\rho_{2}}^{a,b}(s)={\displaystyle \sum_{\g\in G}\exp(-s\cdot d_{\rho_{1},\rho_{2}}^{a,b}(o,\g o))}$
is the Poincaré series of the weighted Manhattan metric $d_{\rho_{1},\rho_{2}}^{a,b}$.

We have a rough picture of the corresponding Manhattan curve ${\cal C}(\rho_{1},\rho_{2})$
for all Fuchsian representations.

\begin{thm}
\label{thm: Manhattan curves are conti} Let $S$ be a surface with
negative Euler characteristic, and let $\rho_{1},\rho_{2}$ be two
Fuchsian representations of $G=\pi_{1}S$ into ${\rm PSL}(2,\R)$.
We denote $S_{1}=\rho_{1}(G)\backslash\H$ and $S_{2}=\rho_{2}(G)\backslash\H$.
Then

\begin{itemize}[font=\normalfont] 

\item[(1)] $(h_{top}(S_{1}),0)$ and $(0,h_{top}(S_{2}))$ are on
${\cal C}(\rho_{1},\rho_{2})$; 

\item[(2)] ${\cal C}(\rho_{1},\rho_{2})$ is convex; and

\item[(3)] ${\cal C}(\rho_{1},\rho_{2})$ is a continuous curve.

\end{itemize}
\end{thm}

\begin{proof}
The first assertion is obvious. The second assertion is because that
the domain 
\[
\{(a,b):\ Q_{\rho_{1},\rho_{2}}^{a,b}(1)<\infty\}
\]
 is convex. To see it is convex, by the Hölder inequality, for $(a_{1},b_{1}),(a_{2},b_{2})\in D$
we have 
\[
Q^{ta_{1}+(1-t)b_{1},ta_{2}+(1-t)b_{2}}(1)\leq(Q^{a_{1},b_{1}}(1))^{t}\cdot(Q^{a_{2},b_{2}}(1))^{1-t}.
\]

To see ${\cal C}$ is continuous, we notice that because ${\cal C}$
is convex, we know ${\cal C}$ is homeomorphic to the straight line
connecting $(h_{top}(S_{1}),0)$ and $(0,h_{top}(S_{2}))$. 
\end{proof}

In the rest of this subsection, we focus on $\rho_{1},\rho_{2}$ being
boundary-preserving isomorphic Fuchsian representations satisfying
the extended Schottky condition. We will see for these representations,
we have much better understanding of the Manhattan curve ${\cal C}(\rho_{1},\rho_{2})$. 

As it is known that for geometrically finite negatively curved manifolds,
the (exponential) growth rate of closed geodesics is exactly the critical
exponent (cf. \cite{Otal:2004fn}), we prove that the critical exponent
$\delta_{\rho_{1},\rho_{2}}^{a,b}$ can also be realized by the growth
rate of hyperbolic elements (or equivalently, closed orbits). To reach
that, inspired by Paulin-Pollicott-Schapira \cite{Pollicott:2012ud},
we introduce several related geometric growth rates. Through analyzing
these growth rates, we are able to link the dynamics critical exponent
$t_{a,b}$ (i.e., the solution to the Bowen's formula) with the geometric
critical exponent $\delta_{\rho_{1},\rho_{2}}^{a,b}$. In result,
these geometric growth rates give us the full picture of the Manhattan
curve ${\cal C}(\rho_{1},\rho_{2})$. 

Recall that for each closed geodesic $\lambda$ on $S_{1}$, it corresponds
a unique geodesic on $S_{2}$, abusing the notation, we still denote
it by $\lambda$. Moreover, $l_{i}[\lambda]$ denotes the length of
the closed geodesic $\lambda$ on $S_{i}$ for $i=1,2$.
\begin{defn}
[Geometric growth rates counted from $S_1$]Let $S$ be a surface
with negative Euler characteristic, and $G:=\pi_{1}S$. Suppose $\rho_{1},\rho_{2}:G\to\psl(2,\R)$
are boundary-preserving isomorphic Fuchsian representations satisfying
the extended Schottky condition. 

\begin{enumerate}[font=\normalfont]

\item $Q_{PPS,x,y}^{a,b}(s):={\displaystyle \sum_{\g\in G}e^{-d^{a,b}(x,\g y)-sd(x,\rho_{1}(\g)y)}}$
is called the \textit{Paulin-Pollicott-Schapira's (PPS) Poincaré series.}

\item $\delta_{PPS}^{a,b}$ is the \textit{critical exponent }of
$Q_{PPS,x,y}^{a,b}(s)$, i.e., $Q_{PPS,x,y}^{a,b}(s)$ converges when
$s>\delta_{PPS}^{a,b}$ and $Q_{PPS,x,y}^{a,b}(s)$ diverges when
$s<\delta_{PPS}^{a,b}$ .

\item $G_{x,y}^{a,b}(s):={\displaystyle \sum_{\g\in G;d(x,\rho_{1}(\g)y)\leq s}e^{-d^{a,b}(x,\g y)}}$
\textit{.}

\item $Z_{W}(s):={\displaystyle \sum_{\stackrel[\lambda\in{\rm Per}_{1}(s)]{\lambda\cap W\neq\phi}{}}e^{-al_{1}[\lambda]-bl_{2}[\lambda]}}$
where $W\subset T^{1}S_{1}$ is a relatively compact open set and
${\rm Per}_{1}(s):=\{\lambda:\ \mbox{\ensuremath{\lambda}\ is a closed orbit on }T^{1}S_{1}{\rm \ and\ }l_{1}[\lambda]\leq s\}$.

\item $P_{Gur}^{ab}:={\displaystyle \limsup_{s\to\infty}\frac{1}{s}\log Z_{W}(s)}$
is the \textit{geometric Gurevich pressure}.

\end{enumerate}
\end{defn}

\begin{lem}
\label{lem:all the same}$\delta_{PPS}^{a,b}=P_{Gur}^{ab}={\displaystyle \lim_{s\to\infty}\frac{1}{s}\log G_{x,y}^{a,b}(s)}={\displaystyle \lim_{s\to\infty}\frac{1}{s}\log Z_{W}(s)}$
for any relative compact $W\subset T^{1}S_{1}$. 
\end{lem}

\begin{proof}
This proof follows the (short) proof of Corollary 4.2, Corollary 4.5
and Theorem 4.7 \cite{Pollicott:2012ud} (also the proof of Theorem
2.4 \cite{Peigne:2013tn}). The strategy is standard but tedious.
We leave the proof in the appendix. 
\end{proof}

Furthermore, we show in below that the geometric Gurevich pressure
$P_{Gur}^{ab}$ matches with the symbolic Gurevich pressure (for the
suspension flow). 

In what follows, $(\TMS,\s)$ stands for countable Markov shift associated
with $\rho_{1}$, $\rho_{2}$ defined in Section \ref{sec:Preliminaries-on-Extended},
and $\tau,\kappa:\TMS\to\R^{+}$ stand for the corresponding geometric
potentials. Recall that $(\Sigma^{+},\s)$ is topologically mixing
and satisfies the BIP property, and $\tau,\kappa$ are locally Hölder
continuous functions and bounded away from zero. Let $\TMS_{\tau}$
be the suspension space relative to $\tau$ and $\phi:\TMS_{\tau}\to\TMS_{\tau}$
be the suspension flow.

We consider a function $\psi:\Sigma_{\tau}^{+}\to\R^{+}$ given by
$\psi(x,t):=\frac{\kappa(x)}{\tau(x)}$ for $x\in\TMS$, $0\leq t\leq\tau(x)$
and $\psi(x,\tau(x))=\psi(\sigma(x),0)$. Using this function $\psi$,
we can reparametrize the suspension flow $\phi:\TMS_{\tau}\to\TMS_{\tau}$
and derive information of orbits of the geodesic flow over $T^{1}S_{2}$.
Roughly speaking, $\psi$ is a reparametrization function, in the
symbolic sense, of the geodesic flow over $T^{1}S_{1}$ such that
the reparametrized flow is conjugated to the geodesic flow over $T^{1}S_{2}$.

\begin{lem}
\label{lem:Pgur_sym<->Pgur_geo}Suppose $\psi:\Sigma_{\tau}^{+}\to\R^{+}$
is defined as $\psi(x,t):=\frac{\kappa(x)}{\tau(x)}$ for $x\in\TMS$,
$0\leq t\leq\tau(x)$ and \textup{$\psi(x,\tau(x))=\psi(\sigma(x),0)$.}
Then $P_{\phi}(-a-b\psi)=P_{Gur}^{ab}.$
\end{lem}

\begin{proof}
Notice that since $S_{1}$ is geometrically finite, there exists a
relatively compact open set $W$ such that $W$ meets every closed
orbit on $T^{1}S_{1}$. Therefore we have for any $g_{0}\in{\cal A}=\{h_{1},...,h_{N_{1}},p_{1},...,p_{N_{2}}\}$
\[
\frac{1}{s}Z_{g_{0}}(s)\leq Z_{W}^{a,b}(s)\leq\sum_{g\in{\cal A}}Z_{g}(s)+C
\]

where $Z_{g}(T)={\displaystyle \sum_{\stackrel[0\leq s\leq T]{\phi_{s}(x,0)=(x,0),}{}}e^{\int_{0}^{s}(-a-b\psi)\circ\phi_{t}(x,t)\dd t}}\chi_{[g]}(x)$
for $g\in{\cal A}$.

The first inequality is because for a closed orbit $\phi_{t}(x,0)=(x,0)$,
$x=g_{0}x_{2}x_{3}...$, $0\leq t\leq s$, of the suspension flow,
it corresponds at most $s$ closed orbits on $T^{1}S_{1}$. The constant
$C$ in the second inequality is from closed geodesics corresponding
to the hyperbolic generators $h_{i}$ (because these closed geodesics
are not in our coding). 

Recall that by definition, we have $P_{\phi}(-a-b\psi)={\displaystyle \lim_{s\to\infty}\frac{1}{s}\log Z_{g_{0}}(s)}$,
and  by Definition \ref{def:-Gurevich pressure for sus flow} 
\[
P_{\phi}(-a-b\psi)=\lim_{s\to\infty}\frac{1}{s}\log Z_{g_{0}}(s),\ {\rm for\ any\ }g_{0}\in{\cal A};
\]
hence $P_{\phi}(-a-b\psi)=P_{Gur}^{ab}.$
\end{proof}

\begin{lem}
\label{lem:pps critical pt<->classical pt}$\delta_{PPS}^{a,b}=0\ {\rm if\ and\ only\ if\ }\delta^{a,b}=1$.
\end{lem}

\begin{proof}
We first notice that the critical exponents are irrelevant with base
points, therefore we can choose 
\[
d_{a,b}(o,\gamma o)=ad(o,\rho_{1}(\gamma)o)+bd(fo,\rho_{2}(\gamma)fo)
\]
where $f:\H\to\H$ is the $\iota-$equivalent bilipschitz given in
Theorem \ref{thm: bilipsthitz } and $\iota:\rho_{1}(G)\to\rho_{2}(G)$
is the boundary-preserving isomorphism. Since $f:\H\to\H$ is bilipschitz,
there exists $C>1$ such that for $\g\in G$ and a fixed $o\in\H$
\[
\frac{1}{C}d(fo,\rho_{2}(\g)fo)\leq d(o,\rho_{1}(\g)o)\leq Cd(fo,\rho_{2}(\g)fo).
\]

With the inequalities above, the desire results are straightforward.
To simplify the notation, in this proof $d(o,\rho_{1}(\g)o)$ is denoted
by $d_{1}(\g)$ and $d(fo,\rho_{2}(\g)fo)$ is denoted by $d_{2}(\g).$

$(\implies)$ Suppose $\delta_{PPS}^{a,b}=0$. 

\textbf{Claim:} 
\[
{\displaystyle \sum_{\g\in G}e^{s(-ad_{1}(\g)-bd_{2}(\g))}<\infty\ \mbox{for \ensuremath{s>1}.}}
\]

pf. Let $s=1+t_{0}$ for some $t_{0}>0$. We have 
\begin{alignat*}{1}
\sum_{\g\in G}e^{s(-ad_{1}(\g)-bd_{2}(\g))} & =\sum_{\g\in G}e^{-ad_{1}(\g)-bd_{2}(\g)+t_{0}(-ad_{1}(\g)-bd_{2}(\g))}\\
 & \leq\sum_{\g\in G}e^{-ad_{1}(\g)-bd_{2}(\g)+t_{0}(-ad_{1}(\g)-b(\frac{1}{C}d_{1}(\g)))}\\
 & =\sum_{\g\in G}e^{-ad_{1}(\g)-bd_{2}(\g)-t_{0}(a+\frac{b}{C})d_{1}(\g)}\\
 & <\infty.
\end{alignat*}

Similarly, we have
\[
{\displaystyle \sum_{\g\in G}e^{s(-ad_{1}(\g)-bd_{2}(\g))}=\infty\ \mbox{for \ensuremath{s<1}.}}
\]

Hence, $\delta^{a,b}=1.$

$(\Longleftarrow)$ Suppose $\delta^{a,b}=1$.

\textbf{Claim:} 
\[
{\displaystyle \sum_{\g\in G}e^{-ad_{1}(\g)-bd_{2}(\g)-td_{1}(\g)}<\infty\ \mbox{for }t>0.}
\]

pf. Recall that there exists $C>1$ such that $\frac{1}{C}d_{1}(\g)<d_{2}(\g)<Cd_{1}(\g)$. 

For any $t>0$, we pick $s_{0}=\frac{a+bC+t}{a+bC}>1$, and we have 

\begin{alignat*}{1}
s_{0}=\frac{a+bC+t}{a+bC} & \iff\frac{-as_{0}+a+t}{s_{0}b-b}=C>\frac{d_{2}}{d_{1}}
\end{alignat*}
which implies 
\[
ad_{1}(\g)+bd_{2}(\g)+td_{1}(\g)>s_{0}(ad_{1}(\g)+bd_{2}(\g));
\]

Because $s_{0}>1=\delta^{a,b}$, we know 
\begin{alignat*}{1}
\sum_{\g\in G}e^{-ad_{1}(\g)-bd_{2}(\g)-td_{1}(\g)} & \leq\sum_{\g\in G}e^{-s_{0}(ad_{1}(\g)+bd_{2}(\g))}<\infty
\end{alignat*}

Similarly, one can show 
\[
{\displaystyle \sum_{\g\in G}e^{-ad_{1}(\g)-bd_{2}(\g)-td_{1}(\g)}=\infty\ \mbox{for }t<0.}
\]
Therefore, $\delta_{PPS}^{a,b}=1$.
\end{proof}

We have an immediate corollary:

\begin{cor}
\label{lem:Gur pressure<->pps critical pt}$P_{\phi}(-a-b\psi)=P_{Gur}^{a,b}=0{\rm \ if\ and\ only\ if\ }\delta^{a,b}=1$.
\end{cor}

\subsection{Proof of Main Results}

Throughout this subsection, $\rho_{1},\rho_{2}$ are boundary-preserving
isomorphic Fuchsian representations satisfying the extended Schottky
condition, and $S_{1}=\rho_{1}(G)\backslash\H$, $S_{2}=\rho_{2}(G)\backslash\H$.
Let $(\TMS,\s)$ be the topologically mixing countable Markov shift
associated with $\rho_{1}$, $\rho_{2}$ defined in Section \ref{sec:Preliminaries-on-Extended},
and $\tau,\kappa:\TMS\to\R^{+}$ be the corresponding geometric potentials.
Recall that $\TMS_{\tau}$ is the suspension space relative to $\tau$
and $\phi:\TMS_{\tau}\to\TMS_{\tau}$ is the suspension flow, and
the reparametrization function $\psi:\Sigma_{\tau}^{+}\to\R^{+}$
is given by $\psi(x,t):=\frac{\kappa(x)}{\tau(x)}$ for $x\in\TMS$,
$0\leq t\leq\tau(x)$ and $\psi(x,\tau(x))=\psi(\sigma(x),0)$.

\begin{lem}
\label{lem:Pressure flow <-> shift} Suppose $\psi:\Sigma_{\tau}^{+}\to\R^{+}$
is defined $\psi(x,t):=\frac{\kappa(x)}{\tau(x)}$ for $x\in\TMS$,
$0\leq t\leq\tau(x)$ and \textup{$\psi(x,\tau(x))=\psi(\sigma(x),0)$.}
Then $P_{\sigma}(-a\tau-b\kappa)=0$ if and only if $P_{\phi}(-a-b\psi)=0.$ 
\end{lem}

\begin{proof}
$(\Longrightarrow)$ Suppose $P_{\sigma}(-a\t-b\k)=0<\infty$. Then
when $t\in(-\varepsilon,\varepsilon)$, $P_{\sigma}(-a\tau-b\k-t\tau)$
is a real analytic and is strictly decreasing, i.e., 
\[
P_{\s}(-a\t-b\k-t\t)\begin{cases}
<0, & {\rm for\ }t>0;\\
=0, & {\rm for\ }t=0;\\
>0, & {\rm for}\ t<0.
\end{cases}
\]

Therefore, by Theorem \ref{thm:Pressure formuals for flow} and $\Delta_{-a-b\psi}=-a\t-b\k$,
we have $P_{\phi}(-a-b\psi)=0$.

$(\Longleftarrow)$ To see $P_{\phi}(-a-b\psi)=0{\rm \ implies\ }P_{\s}(-a\t-b\k)=0$
Notice that because $\tau>c>0$ implies $\sum_{i=0}^{\infty}\tau\circ\sigma^{i}=\infty$,
by Lemma 4.1 and Remark 4.1 in Jaerisch-Kesseböhmer-Lamei \cite{Jaerisch:2014js},
we have 

\begin{alignat*}{1}
0= & P_{\phi}(-a-b\psi)\\
= & \sup\left\{ \frac{h_{\sigma}(\mu)}{\int\tau\dd\mu}+\frac{\int\Delta_{-a-b\psi}\dd\mu}{\int\t\dd\mu}:\ \mu\in{\cal M}_{\sigma}(\tau)\mbox{ with\ }\Delta_{-a-b\psi}\in L^{1}(\mu)\right\} \\
= & \sup\left\{ \frac{h_{\sigma}(\mu)}{\int\tau\dd\mu}+\frac{\int(-a\t-b\k)\dd\mu}{\int\t\dd\mu}:\ \mu\in{\cal M}_{\sigma}(\tau)\mbox{ with\ }-a\t-b\k\in L^{1}(\mu)\right\} 
\end{alignat*}
where ${\cal M}_{\s}(\t):=\left\{ \mu:\ \mu\in{\cal M}_{\sigma}\ {\rm and}\ \int\tau\dd\mu<\infty\right\} $. 

For all $\mu\in{\cal M}_{\sigma}$ such that $-a\tau-b\k\in L^{1}(\mu)$,
we have $\int\tau\dd\mu>c>0$; hence, 
\[
0=\sup\left\{ h_{\sigma}(\mu)+\int(-a\t-b\k)\dd\mu:\ \mu\in{\cal M}_{\sigma}(\tau)\ {\rm and}\ -a\tau-b\k\in L^{1}(\mu)\right\} .
\]

Recall that 
\[
P_{\sigma}(-a\tau-b\kappa)=\sup\left\{ h_{\sigma}(\mu)+\int(-a\t-b\k)\dd\mu:\ \mu\in{\cal M}_{\sigma}\ {\rm and}\ -a\tau-b\k\in L^{1}(\mu)\right\} .
\]

Notice that for $\mu\in{\cal M_{\sigma}}$, if $-a\tau-b\k\in L^{1}(\mu)$
then $\int\tau\dd\mu<\infty$ (i.e., $\mu\in{\cal M}_{\sigma}(\tau)$).

Moreover, it is obvious that ${\cal M_{\s}}(\tau)\subset{\cal M}_{\s}$.
Thus, we have 
\begin{alignat*}{1}
P_{\sigma}(-a\tau-b\kappa)= & \sup\left\{ h_{\sigma}(\mu)+\int(-a\t-b\k)\dd\mu:\ \mu\in{\cal M}_{\sigma}\ {\rm and}\ -a\tau-b\k\in L^{1}(\mu)\right\} \\
= & \sup\left\{ h_{\sigma}(\mu)+\int(-a\t-b\k)\dd\mu:\ \mu\in{\cal M}_{\sigma}(\tau)\ {\rm and}\ -a\tau-b\k\in L^{1}(\mu)\right\} \\
= & 0
\end{alignat*}

\end{proof}

The following theorem gives more geometric characterizations to $t_{a,b}$
(i.e., the solution to the equation $P_{\sigma}(-t_{a,b}(a\tau+b\kappa))=0$).
Without any surprise, as the famous Bowen's formula, $t_{a,b}$ is
indeed the critical exponent $\delta^{a,b}$ and the growth rate of
hyperbolic elements.

\begin{thm}
[Bowen's formula]\label{thm:t_ab=00003Ddelta_ab}For $(a,b)\in D$.
Suppose $t_{a,b}$ is the solution to $P_{\sigma}(-t_{a,b}(a\tau+b\kappa))=0$.
Then 
\[
t_{a,b}=\delta^{a,b}=\lim_{s\to\infty}\frac{1}{s}\log\overline{G}_{x,y}^{a,b}(s)
\]
 where $\overline{G}_{x,y}^{a,b}(s):=\#\{\gamma\in G:\ d^{a,b}(x,\g y)\leq s\}$;
\end{thm}

\begin{proof}
We first notice that 
\begin{alignat*}{2}
\delta^{a,b}=1 & \iff\delta_{PPS}^{a,b}=0 & {\rm Lemma\ }\ref{lem:pps critical pt<->classical pt}\ \\
 & \iff P_{Gur}^{a,b}=0 & {\rm Lemma\ }\ref{lem:all the same}\ \\
 & \iff P_{\phi}(-a-b\psi)=0 & {\rm Lemma}\ \ref{lem:Pgur_sym<->Pgur_geo}\ \\
 & \iff P_{\sigma}(-a\tau-b\kappa)=0 & \ \ \ \ \ \ {\rm Lemma\ }\ref{lem:Pressure flow <-> shift}.
\end{alignat*}

Thus, $P_{\sigma}(-t_{a,b}(a\tau+b\kappa))=0$ if and only $\delta^{t_{a,b}a,t_{a,b}b}=1$,
that is, $Q^{t_{a,b}a,t_{a,b}b}(s)={\displaystyle \sum_{\g\in G}e^{-t_{ab}d^{a,b}(o,\g o)}}$
has critical exponent 1. Hence, $Q^{a,b}(s)={\displaystyle \sum_{\gamma\in G}e^{-sd^{a,b}(o,\g o)}}$
has critical exponent $t_{a,b}$, i.e., $\delta^{a,b}=t_{a,b}$.

For the rear inequality, the prove is the same as the proof of Lemma
\ref{lem:all the same} with some simplification (in other words,
the proof is a modification of Lemma 3.3, Corollary 4.5, Theorem 4.7
\cite{Pollicott:2012ud}, or Section 2.2 \cite{Peigne:2013tn}). However,
for the completeness, we put the proof in the appendix.
\end{proof}

\begin{rem}
Using the same argument as in Lemma \ref{lem:all the same}, one can
also prove that the critical exponent $\delta^{a,b}$ is the growth
rate of closed geodesics on $S_{1}$ and $S_{2}$. One notices that
each closed geodesic on $S_{1}$ (and $S_{2}$ )is corresponds to
a hyperbolic element in $\G_{1}$ (and $\G_{2}$). In other words,
\[
\delta^{a,b}=h^{a,b}:=\lim_{s\to\infty}\frac{1}{s}\#\{\g\in G:\ \g\ {\rm is\ hyperblic\ and\ }al_{1}[\g]+bl_{2}[\g]\leq s\}.
\]

\end{rem}

\begin{lem}
\label{lem:man cur=00003D press} The Manhattan curve ${\cal C}(\rho_{1},\rho_{2})$
is the set of solutions to $P_{\sigma}(-a\tau-b\kappa)=0$ in $D$.
\end{lem}

\begin{proof}
It follows from the same argument as the above theorem:
\begin{alignat*}{2}
(a,b)\in{\cal C}(\rho_{1},\rho_{2}) & \iff\delta^{a,b}=1 & {\rm by\ definition}\\
 & \iff\delta_{PPS}^{a,b}=0 & {\rm Lemma\ }\ref{lem:pps critical pt<->classical pt}\ \\
 & \iff P_{Gur}^{a,b}=0 & {\rm Lemma\ }\ref{lem:all the same}\ \\
 & \iff P_{\phi}(-a-b\psi)=0 & {\rm Lemma}\ \ref{lem:Pgur_sym<->Pgur_geo}\ \\
 & \iff P_{\sigma}(-a\tau-b\kappa)=0 & \ \ \ \ \ \ {\rm Lemma\ }\ref{lem:Pressure flow <-> shift}.
\end{alignat*}

\end{proof}

\begin{thm}
\label{Thm: Manhattan cuve analytic} The Manhattan curve ${\cal C}(\rho_{1},\rho_{2})$
is real analytic.
\end{thm}

\begin{proof}
It is a direct consequence of Theorem \ref{thm:P(-at-bk)=00003D0 analytic}
and Lemma \ref{lem:man cur=00003D press}.
\end{proof}

\begin{prop}
\label{prop: Manhattan curve features} Let $\rho_{1},\rho_{2}$ be
two boundary-preserving isomorphic Fuchsian representations satisfying
the extended Schottky condition, and $S_{1}=\rho_{1}(G)\backslash\H$,
$S_{2}=\rho_{2}(G)\backslash\H$. Then

\begin{itemize}[font=\normalfont] 

\item[(1)] $\mathcal{C}$ is strictly convex if $S_{1}$ and $S_{2}$
are NOT conjugate in ${\rm PSL}(2,\R)$; and

\item[(2)] $\mathcal{C}$ is a straight line if and only if $S_{1}$
and $S_{2}$ are conjugate in ${\rm PSL}(2,\R)$.

\end{itemize}\end{prop}
\begin{proof}
This result is a direct consequence of Theorem \ref{thm: Manhattan curves are conti}
and Theorem \ref{Thm: Manhattan cuve analytic}. Indeed, the strictly
convexity comes from the analyticity and the convexity of ${\cal C}$.

It is clear that when $S_{1}$ and $S_{2}$ are isometric we have
${\cal C}$ is a straight line. Conversely, suppose ${\cal C}$ is
a straight line. Then the slope of the tangent line of the Manhattan
curve ${\cal C}$ is a constant, i.e., 
\[
b'=-\frac{h_{top}(S_{2})}{h_{top}(S_{1})}=\frac{-\int\tau\dd m_{-a\tau-b(a)\kappa}}{\int\kappa\dd m_{-a\tau-b(a)\kappa}}
\]
 where $m_{-a\tau-b(a)\kappa}$ is the equilibrium state for $-a\tau-b(a)\kappa$
for all $a\in[0,h_{top}(S_{1})]$. In particular, 
\[
b'=-\frac{\int\tau\dd m_{-h_{top}(S_{1})\tau}}{\int\kappa\dd m_{-h_{top}(S_{1})\tau}}=-\frac{\int\tau\dd m_{-h_{top}(S_{2})\kappa}}{\int\kappa\dd m_{-h_{top}(S_{2})\kappa}}.
\]

\textbf{Claim:} $h_{top}(S_{1})\tau$ and $h_{top}(S_{2})\kappa$
are cohomologus.

It is clear that we have the desired result after we prove the claim.
Because $h_{top}(S_{1})\tau\sim h_{top}(S_{2})\kappa$ means that
$S_{1}$ and $S_{2}$ have proportional marked length spectra. Then
by the proportional marked length spectrum rigidity (i.e., Theorem
\ref{thm:proportinal marked length spectrum}) we are done.

pf. for short, we denote $m_{1}=m_{-h_{top}(S_{1})\tau}$ and $m_{2}=m_{-h_{top}(S_{2})\kappa}$.
We prove this claim by the uniqueness of the equilibrium states. In
other words, we want to show that $m_{2}$ is the equilibrium state
for $-h_{top}(S_{1})\tau$, that is, 
\[
0=P_{\sigma}(-h_{top}(S_{1})\tau)=h(m_{2})-h_{top}(S_{1})\int\tau\dd m_{2}.
\]

Notice that, by definition, 
\[
0=P_{\sigma}(-h_{top}(S_{2})\kappa)=h(m_{2})-h_{top}(S_{2})\int\kappa\dd m_{2}
\]
and, by the above observation, 
\[
\frac{h_{top}(S_{1})}{h_{top}(S_{2})}=\frac{\int\kappa\dd m_{2}}{\int\tau\dd m_{2}}.
\]
Thus, we have
\begin{alignat*}{1}
h(m_{2})-h_{top}(S_{1})\int\tau\dd m_{2} & =h_{top}(S_{2})\int\kappa\dd m_{2}-h_{top}(S_{1})\int\tau\dd m_{2}\\
 & =0\\
 & =P_{\sigma}(-h_{top}(S_{1})\tau).
\end{alignat*}

By the uniqueness of the equilibrium states (cf. Theorem \ref{thm: Phase transition}),
we know $m_{1}=m_{2}$. Moreover, Theorem 4.8 \cite{Sarig:2009wta}
showed that this only happens when $-h_{top}(S_{1})\tau$ and $-h_{top}(S_{2})\kappa$
are cohomologous. 
\end{proof}

\begin{cor}
[Bishop-Steger's entropy rigidity \cite{Bishop:1991gz}]\label{cor:-Bishop-Steger}
Let $\rho_{1},\rho_{2}$ be two boundary-preserving isomorphic Fuchsian
representations satisfying the extended Schottky condition, and $S_{1}=\rho_{1}(G)\backslash\H$,
$S_{2}=\rho_{2}(G)\backslash\H$. Then, for any fixed $o\in\H$, 
\[
\delta^{1,1}=\lim_{T\to\infty}\frac{1}{T}\log\#\{\g\in G:\ d(o,\rho_{1}(\g)o)+d(o,\rho_{2}(\g)o)\leq T\}.
\]
Moreover, $\delta^{1,1}\leq\dfrac{h_{top}(S_{1})\cdot h_{top}(S_{2})}{h_{top}(S_{1})+h_{top}(S_{2})}$
and the equality holds if and only if $S_{1}$ and $S_{2}$ are isometric.
\end{cor}

\begin{proof}
By Theorem \ref{thm:t_ab=00003Ddelta_ab}, we know $\delta^{1,1}(1,1)\in{\cal C}$
is the intersection of ${\cal C}$ and the line $a=b$. By the convexity
of ${\cal C}$, we know that the intersection of the line $a=b$ and
$b=\frac{-h_{top}(S_{2})}{h_{top}(S_{1})}a+h_{top}(S_{2})$ lies above
$\delta^{1,1}(1,1)$. See the following picture. \begin{center} \includegraphics[scale=0.8]{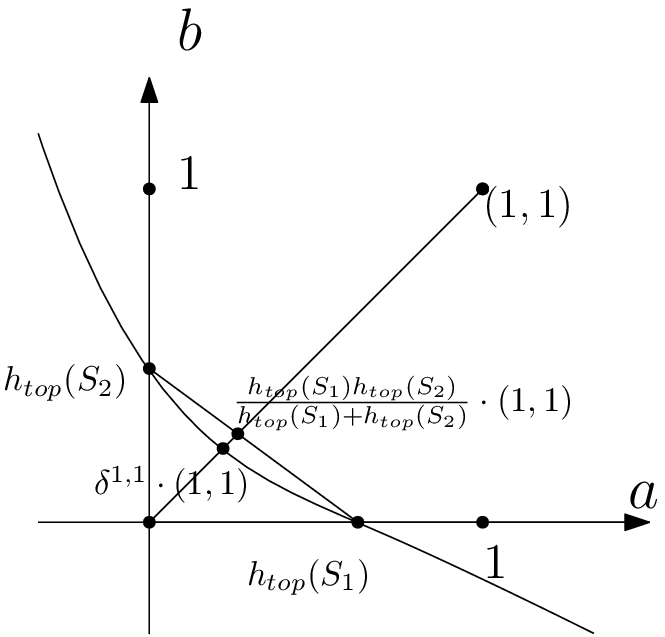}\end{center} 

Therefore, we have $\delta^{1,1}\leq\dfrac{h_{top}(S_{1})\cdot h_{top}(S_{2})}{h_{top}(S_{1})+h_{top}(S_{2})}$
. Moreover, when the equality holds, we have ${\cal C}$ is a straight
line. By Proposition \ref{prop: Manhattan curve features}, we are
done.
\end{proof}

\begin{defn*}
[Thurston's intersection number, Definition \ref{def: Thurston's-intersection-nb}]Let
$S_{1}$ and $S_{2}$ be two Riemann surfaces. Thurston's intersection
number ${\rm I}(S_{1},S_{2})$ of $S_{1}$ and $S_{2}$ is given by
\[
{\rm I}(S_{1},S_{2})=\lim_{n\to\infty}\frac{l_{2}[\g_{n}]}{l_{1}[\g_{n}]}
\]
 where $\{[\g_{n}]\}_{n=1}^{\infty}$ is a sequence of conjugacy classes
for which the associated closed geodesics $\g_{n}$ become equidistributed
on $\G_{1}\backslash\H$ with respect to area.
\end{defn*}

\begin{cor}
[Thurston's rigidity] \label{cor:Thurston's rigidity} Let $\rho_{1},\rho_{2}$
be two boundary-preserving isomorphic Fuchsian representations satisfying
the extended Schottky condition, and $S_{1}=\rho_{1}(G)\backslash\H$,
$S_{2}=\rho_{2}(G)\backslash\H$. Then ${\rm I}(S_{1},S_{2})\geq\frac{h_{top}(S_{1})}{h_{top}(S_{2})}$
and the equality hold if and only if $\rho_{1}$ and $\rho_{2}$ are
conjugate in ${\rm PSL}(2,\R)$.
\end{cor}

\begin{proof}
It is enough to show that the normal of the tangent of ${\cal C}(S_{1},S_{2})$
at $(h_{top}(S_{1}),0)$ is $I(S_{1},S_{2})$.

Recall that 
\[
b'(a)=\frac{-\int\tau\dd m}{\int\kappa\dd m}
\]
 where $m=m_{-a\tau-b\kappa}$ is the equilibrium state of $-a\tau-b\kappa$.
So, for $a=h_{top}(S_{1})$, $b=0$ we have 
\[
b'(-h_{top}(S_{1}))=-\frac{\int\tau\dd m_{-h_{top}(S_{1})\tau}}{\int\kappa\dd m_{-h_{top}(S_{1})\tau}}.
\]

Thus, it is sufficient to show 
\[
{\rm I}(S_{1},S_{2}):={\displaystyle \lim_{T\to\infty}\frac{{\displaystyle \sum_{\lambda\in{\rm Per}_{1}(T)}l_{2}[\lambda]}}{{\displaystyle \sum_{\lambda\in{\rm Per}_{1}(T)}l_{1}[\lambda]}}}=\frac{\int\kappa\dd m_{-h_{top}(S_{1})\tau}}{\int\tau\dd m_{-h_{top}(S_{1})\tau}}.
\]

Because $m_{-h_{top}(S_{1})\tau}$ is the Bowen-Margulis measure for
the geodesic flow on $T^{1}S_{1}$, and $S_{1}$ is geometrically
finite, we know the Bowen-Margulis measure is equidistributed with
respect to closed orbits (see, for example, Theorem 4.1.1 \cite{Roblin:2003vz}).
Therefore, the above equation is true.
\end{proof}

%% file: 5.Appendix.tex
\section{Appendix\label{sec:Appendix}}

Recall our notation that $\rho_{1},\rho_{2}$ are two boundary-preserving
isomorphic Fuchsian representations satisfying the extended Schottky
condition, and $S_{1}=\rho_{1}(G)\backslash\H$, $S_{2}=\rho_{2}(G)\backslash\H$.
Let $d_{\rho_{1},\rho_{2}}^{a,b}$ be the weighted Manhattan metric.
Recall that $\delta^{a,b}$ is the critical exponent of the Poincaré
series associated with $d_{\rho_{1},\rho_{2}}^{a,b}$.

\subsection*{The Proof of Lemma \ref{lem:all the same}}

We first recall three useful lemmas.

\begin{lem}
[Lemma 2.2 \cite{Schapira:2004a}]\label{lem:not_too_far} Suppose
$a,b,c\in\H$ and $d(a,b)+d(a,c)-d(b,c)\leq C$ for some $C>0$, then
$a$ is in a $D-$neighborhood of the geodesic segment $[b,c]$ where
$D$ is a constant only depending on $C$.
\end{lem}

\begin{lem}
[Lemma 4.4 \cite{Pollicott:2012ud}] \label{lem:super-multiplicativity}
Let $b_{n}\geq0$ such that there exist $C>0$ and $N\in\N$ such
that for all $n,m\in\N$, we have 
\[
b_{n}b_{m}\leq C\sum_{i=-N}^{i=N}b_{n+m+i},
\]
 then with $a_{n}=\sum_{k=0}^{n-1}b_{n}$, the limit of $a_{n}^{\frac{1}{n}}$
as $n\to\infty$ exists (and hence is equal to its limit-sup). 
\end{lem}

Recall that $\delta_{PPS}^{a,b}$ is the critical exponent of $Q_{PPS,x,y}^{a,b}(s)$,
i.e., $Q_{PPS,x,y}^{a,b}(s)$ converges when $s>\delta_{PPS}^{a,b}$
and $Q_{PPS,x,y}^{a,b}(s)$ diverges when $s<\delta_{PPS}^{a,b}$
where ${\displaystyle Q_{PPS,x,y}^{a,b}(s)={\displaystyle \sum_{\g\in G}e^{-d^{a,b}(x,\g y)-sd(x,\rho_{1}(\g)y)}}}$.

W.l.o.g., we can write $d^{a,b}(x,\g y)=ad(x,\g y)+bd(fx,\iota(\g)fy)$
for $\g\in\G_{1}$ and $\iota:\G_{1}\to\G_{2}$ is a boundary-preserving
isomorphism and $f:\H\to\H$ is the bilipschitz map given by Theorem
\ref{thm: bilipsthitz }. To simply our notation, we denote $d_{1}(x,\g y):=d(x,\g y)$
and $d_{2}(x,\g y):=d(fx,\iota(\rho_{1}(\g))\cdot fy)$. Therefore,
$G_{x,y}^{a,b}(s)$ can be equivalently defined as:

\[
G_{x,y}^{a,b}(s):={\displaystyle \sum_{\g\in\G_{1};d_{1}(x,\g y)\leq s}e^{-d^{a,b}(x,\g y)}}.
\]
Similarly, the PPS Poincaré series $Q_{PPS,x,y}^{a,b}(s)$ can be
rewrite as 
\[
Q_{PPS,x,y}^{a,b}(s)=\sum_{\g\in\G_{1}}e^{-d^{a,b}(x,\g y)-sd_{1}(x,\g y)}.
\]

Let us first define several useful growth rates. 

\begin{itemize}[font=\normalfont]

\item $G_{x,y,1}^{a,b}(s):={\displaystyle \sum_{\g\in\G_{1};s-1<d_{1}(x,\g y)\leq s}e^{-d^{a,b}(x,\g y)}}.$

\item$A_{x,y,U'}(s):=\{\g\in\G_{1}:\ d_{1}(x,\g y)\leq s\ {\rm and}\ \g y\in U'\}$
where $U'$ is an open set in $\vbdy\H\times\H$.

\item $a_{x,y,U'}(s):={\displaystyle \sum_{\g\in A_{x,y,U'}(s)}e^{-d^{a,b}(x,\g y)}}$
.

\item$B_{x,y,U',V'}(s):=\{\g\in\G_{1}:\ d_{1}(x,\g y)\leq s\ ,\g y\in U'\ {\rm and}\ \g^{-1}x\in V'\}$
where $U'$, $V'$ are an open sets in $\vbdy\H\times\H$.

\item $b_{x,y,U',V'}(s):={\displaystyle \sum_{\g\in B_{x,y,U',V'}(s)}e^{-d^{a,b}(x,\g y)}}$. 

\end{itemize}

We notice that, by the triangle inequality, we know ${\displaystyle \limsup_{s\to\infty}}\frac{1}{s}\log a_{x,y,U'}^{a,b}(s)$,
${\displaystyle \limsup_{s\to\infty}}\frac{1}{s}\log b_{x,y,U',V'}^{a,b}(s)$
and ${\displaystyle \limsup_{s\to\infty}\frac{1}{s}\log G_{x,y}^{a,b}(s)}$
are independent with the choice of bases point $x$ and $y$, and
it is obvious that $b_{x,y,U',V'}^{a,b}(s)\leq a_{x,y,U'}^{a,b}(s)\leq G_{x,y}^{a,b}$.

\begin{lem*}
[Lemma \ref{lem:all the same}]$\delta_{PPS}^{a,b}=P_{Gur}^{ab}={\displaystyle \lim_{s\to\infty}\frac{1}{s}\log G_{x,y}^{a,b}(s)}={\displaystyle \lim_{s\to\infty}\frac{1}{s}\log Z_{W}(s)}$
for any relative compact $W\subset T^{1}S_{1}$. 
\end{lem*}

The proof of this lemma will be separated into several lemmas. The
following proofs are using the same argument as Lemma 4,2, Corollary
4.5 and Theorem 4.7 \cite{Pollicott:2012ud} with minor modifications.

\begin{lem}
\label{lem:infG-delta-supG}
\[
\liminf_{s\to\infty}\frac{1}{s}\log G_{x,y,1}^{a,b}(s)\leq\delta_{PPS}^{a,b}\leq\limsup_{s\to\infty}\frac{1}{s}\log G_{x,y,1}^{a,b}(s)
\]

\end{lem}

\begin{proof}
The proof is elementary. However, for the completeness, we give a
proof here. We first notice that 
\[
Q_{PPS,x,y}^{a,b}(t)=\sum_{n=0}^{\infty}\sum_{\g\in E_{n}}e^{-d^{a,b}(x,\g y)-td_{1}(x,\g y)}
\]
where $E_{n}:=\{\g\in\G_{1}:\ n-1<d_{1}(x,\g y)\leq n\}$. Therefore,
we have 
\[
\sum_{n}e^{-tn}G_{x,y,1}^{a,b}(n)\leq Q_{PPS,x,y}^{a,b}(t)\leq\sum_{n}e^{-t(n-1)}G_{x,y,1}^{a,b}(n).
\]

\textbf{Claim:} If $t>\Delta:={\displaystyle \limsup_{s\to\infty}\frac{1}{s}\log G_{x,y,1}^{a,b}(s)}$
then $Q_{PPS,x,y}^{a,b}(t)$ converges (i.e., $\delta^{a,b}\leq\Delta).$

pf. For $\varepsilon=\frac{t-\Delta}{2}$, there exists $N>0$ such
that for all $n>N$ we have $\Delta+\varepsilon>\frac{1}{n}\log G_{x,y,1}^{a,b}(n)$.
Therefore, we have 
\begin{alignat*}{1}
Q_{PPS,x,y}^{a,b}(t) & \leq\sum_{n}e^{-t(n-1)}G_{x,y,1}^{a,b}(n)\\
 & <C+\sum_{n=N}^{\infty}e^{-t(n-1)}e^{n(\Delta+\varepsilon)}\\
 & =C+e^{t}\sum_{n=N}^{\infty}e^{n(-t+\Delta+\varepsilon)}<C+e^{t}\sum_{n=N}^{\infty}e^{n(\frac{-t+\Delta}{2})}<\infty
\end{alignat*}
where $C={\displaystyle \sum_{n=0}^{N-1}}e^{-t(n-1)}G_{x,y,1}^{a,b}(n)<\infty$.

\textbf{Claim:} If $t<\overline{\Delta}:={\displaystyle \liminf_{s\to\infty}\frac{1}{s}\log G_{x,y,1}^{a,b}(s)}$
then $Q_{PPS,x,y}^{a,b}(t)$ diverges (i.e., $\delta^{a,b}\geq\overline{\Delta})$.

For $\varepsilon=\frac{\overline{\Delta}-t}{2}$, there exists $N'>0$
such that for all $n>N'$ we have $\overline{\Delta}-\varepsilon<\frac{1}{n}\log G_{x,y,1}^{a,b}(n)$.
Therefore, we have 
\begin{alignat*}{1}
Q_{PPS,x,y}^{a,b}(t) & \geq\sum_{n}e^{-tn}G_{x,y,1}^{a,b}(n)\\
 & \geq\sum_{n=N'}^{\infty}e^{-tn}e^{n(\overline{\Delta}-\varepsilon)}\\
 & =\sum_{n=N'}^{\infty}e^{n(-t+\overline{\Delta}-\varepsilon)}>\sum_{n}e^{n(\frac{-t+\overline{\Delta}}{2})}=\infty.
\end{alignat*}

\end{proof}

\begin{lem}
\label{lem:orbit-growth-limit}The inequalities of the above lemma
are indeed equalities. Moreover, 
\[
\delta_{PPS}^{a,b}=\lim_{n\to\infty}\frac{1}{n}\log G_{x,y}^{a,b}(s).
\]

\end{lem}

\begin{proof}
The proof of this Lemma follows the idea of the (short) proof of Lemma
4.2 (see also the proof of Theorem 2.4 \cite{Peigne:2013tn}). 

We notice that by the triangle inequality, it is obvious that the
$\limsup_{s\to\infty}\frac{1}{s}\log G_{x,y}^{a,b}(s)$ does not depend
on the reference point $x$ and $y$. W.l.o.g, we pick $x=y=o$. Recall
the generating set of the extended Schottky group $G=\pi_{1}S$ is
${\cal A^{\pm}}=\{h_{1}^{\pm},..,h_{N_{1}}^{\pm},p_{1},...,p_{N_{2}}\}$
with $N_{1}+N_{2}\geq3$. 

Let \begin{itemize}

\item$E_{n}:=\{\g\in\G_{1}:\ n-1<d_{1}(o,\g o)\leq n\}$.

\item $b_{n}:=G_{x,y,1}^{a,b}(n)={\displaystyle \sum_{\g\in E_{n}}e^{-d^{a,b}(o,\g o)}}.$

\end{itemize}

By Lemma \ref{lem:super-multiplicativity}, it is enough prove that
there exist $M>0$ and $N\in\N$ such that for all $n,m\in\N$, we
have 
\[
b_{n}b_{m}\leq M\sum_{i=-N}^{i=N}b_{n+m+i}.
\]

\textbf{Claim: }There exist $N\in\N$ and $M>0$ such that $\#E_{n}\times\#E_{m}\leq M\cdot\sum_{i=-N}^{i=N}\#E_{n+m+i}$

pf. Let $\g_{n}\in E_{n}$ and $\g_{m}\in E_{m}$, by Lemma \ref{lem:d(x,y)>d(x,o)+d(y,0)+C},
there exists $\alpha\in{\cal A^{\pm}}$ (more precisely, if $\g_{n}=g_{i}...$
and $\g_{m}=g_{j}...$ for $g_{i},g_{j}\in{\cal A}$ then we take
$\alpha=g_{k}$ for $g_{k}\in{\cal A^{\pm}}\backslash\{g_{i}^{\pm},g_{j}^{\pm}\}$)
such that 
\[
\left|d(o,\g_{n}\rho_{1}(\alpha)\g_{m}o)-d(o,\g_{n}o)-d(o,\g_{m}o)\right|<C_{1}
\]
 and 
\[
\left|d(o,(\iota\circ\g_{n})\rho_{2}(\alpha)(\iota\circ\g_{m})o)-d(o,(\iota\circ\g_{n})o)-d(o,(\iota\circ\g_{m})o)\right|<C_{2}
\]
where $C_{1}$ only depending on $\rho_{1}$ and $C_{2}$ only depending
on $\rho_{2}$ .

Thus, 
\[
n+m-C_{1}-2<d(o,\g_{n}\rho_{1}(\alpha)\g_{m}o)\leq n+m+C_{1}+2.
\]
Let us consider the map 
\begin{align*}
\Psi:E_{n}\times E_{m} & \to\sum_{i=-C_{1}-2}^{i=C_{1}+2}\#E_{n+m+i}\\
(\g_{n},\g_{m}) & \mapsto\g_{n}\rho_{1}(\alpha)\g_{m}
\end{align*}

This maps is obvious not one-to-one. Nevertheless, we claim $\#\Psi^{-1}(\g_{n}\rho_{1}(\alpha)\g_{m})$
is finite. By Lemma \ref{lem:not_too_far}, we know that $d(\g_{n}o,[o,\g_{n}\rho_{1}(\alpha)\g_{m}o])\leq D$
(where $D$ only depends on $C_{1}$), and which implies if there
exist $\g_{n}'\in E_{n}$ and $\g_{m}'\in E_{m}$ such that $\g'_{n}\rho_{1}(\alpha)\g_{m}'=\g{}_{n}\rho_{1}(\alpha)\g_{m}=\g$
then $d(\g_{n}o,\g_{n}'o)\leq2(D+1)$ (because $n-1<d(\g_{n}o,o),d(\g'_{n}o,o)\leq n$,
and $\g_{n}o$, $\g'_{n}o$ are in a $D$-neighborhood of $[o,\g o]$).
Moreover, by the discreteness of $\G_{1}$, the set $\{\g\in\G_{1}:\ d(\g o,o)\leq2(D+1)\}$
is finite (say, smaller than or equal to $M_{1}$). Hence $\#\Psi^{-1}(\g_{n}\rho_{1}(\alpha)\g_{m})\leq M_{1}^{2}$.

Therefore, 
\[
\#E_{n}\times\#E_{m}\leq(2N_{1}+N_{2})M_{1}^{2}\cdot\sum_{i=-C_{1}-2}^{i=C_{1}+2}\#E_{n+m+i}
\]
 where $2N_{1}+N_{2}$ is the cardinality of ${\cal A}^{\pm}$. 

Moreover, we know 
\[
\left|d^{a,b}(o,\g_{n}\rho_{1}(\alpha)\g_{m}o)-d^{a,b}(o,\g_{n}o)-d^{a,b}(o,\g_{m}o)\right|\leq aC_{1}+bC_{2},
\]
thus we have the lemma, more precisely, 
\[
b_{n}b_{m}\leq(N_{1}+N_{2})M_{1}^{2}\cdot e^{aC_{1}+bC_{2}}\sum_{i=-(C_{1}+2)}^{i=C_{1}+2}b_{n+m+i}.
\]

\end{proof}

\begin{lem}
\label{lem:pps_a}
\[
{\displaystyle \lim_{s\to\infty}\frac{1}{s}\log a_{x,y,U'}^{a,b}(s)}=\delta_{PPS}^{a,b}.
\]
\end{lem}
\begin{proof}
It is obvious that $a_{x,y,U'}^{a,b}(s)\leq G_{x,y}^{a,b}(s)$, so
it is enough to prove 
\[
{\displaystyle \liminf_{s\to\infty}\frac{1}{s}\log a_{x,y,U'}^{a,b}(s)}\geq\delta_{PPS}^{a,b}.
\]

By the compactness of $\L_{1}$, there exist $\g_{1},...,\g_{k}\in\G_{1}$
such that $\L_{1}\subset{\displaystyle \bigcup_{i=1}^{k}\gamma_{k}U'}$.
Since ${\rm Conv}(\L_{1})\backslash\bigcup_{i=1}^{k}\gamma_{k}U'$
is compact, there exists a constant $c_{1}'\geq0$ such that 
\begin{equation}
G_{x,y}^{a,b}(s)\leq c_{1}'+\sum_{i=1}^{k}a_{x,y,\g_{i}U'}(s).\label{eq:G<A}
\end{equation}

\textbf{Claim:} There exists a constant $c>0$ such that $a_{x,y,\g_{i}U'}(s-r)\leq e^{c}a_{x,y,U'}(s)$
for all $s>r$ where $r=\max\{d_{1}(x,\g_{i}x):\ i\in1,2..,k\}$.

It is clear that using this claim and (\ref{eq:G<A}), we have 
\[
a_{x,y,U'}(s)\geq\frac{1}{ke^{c}}(G_{x,y}^{a,b}(s-r)-c_{1}'),
\]
 and, thus, the lemma.

pf of the claim: We first notice that by definition we have 
\[
a_{x,y,U'}(s)=a_{\g x,y,\g U'}(s).
\]
 To be more precise, it is because $\g^{-1}\left(A_{\g x,y,\g U'}(s)\right)$$=A_{x,y,U'}(s)$,
and also $d^{a,b}(x,gy)=d^{a,b}(\g x,\g gy)$ for all $g\in A_{x,y,U'}(s)$. 

Therefore, it is enough to show there exists $c>0$ such that $a_{x,y,\g_{i}U'}(s-r)=a_{\g_{i}^{-1}x,y,U'}(s-r)\leq e^{c}a_{x,y,U'}(s)$.

To see that, we notice that by the triangle inequality we have if
$d_{1}(\g_{i}^{-1}x,\g y)\leq s-r$ then $d_{1}(x,\g y)\leq s$. Thus,

\begin{alignat*}{1}
A_{\g_{i}^{-1}x,y,U'}(s-r)= & \{g\in\G_{1}:\ d_{1}(\g_{i}^{-1}x,gy)\leq s-r\ {\rm and}\ gy\in U'\}\\
\subset A_{x,y,U'}(s) & =\{g\in\G_{1}:\ d_{1}(x,gy)\leq s\ {\rm and}\ gy\in U'\}.
\end{alignat*}

Furthermore, since $d^{a,b}$ satisfies the triangle inequality, we
have 
\begin{alignat*}{1}
\left|-ad_{1}(x,\g y)-bd_{2}(x,\g y)-(-ad_{1}(\g_{i}^{-1}x,\g y)-bd_{2}(\g_{i}^{-1}x,\g y))\right|\\
\leq & ad_{1}(x,\g_{i}^{-1}x)+bd_{2}(x,\g_{i}^{-1}x)\\
\leq & c_{2}'(x,a,b).
\end{alignat*}
Hence $a_{x,y,\g_{i}U'}(s-r)\leq e^{c_{2}'}a_{x,y,U'}(s)$.
\end{proof}

\begin{lem}
\label{lem:pps_b}
\[
\lim_{s\to\infty}\frac{1}{s}\log b_{x,y,U',V'}^{a,b}(s)=\delta_{PPS}^{a,b}.
\]
\end{lem}
\begin{proof}
Similar to proof of the previous lemma. There exist $\alpha_{1},...,\alpha_{l}\in\G_{1}$
such that $\L_{1}\subset{\displaystyle \bigcup_{i=1}^{k}\alpha_{i}V'}$.
Therefore, there exists a constant $c_{3}'\geq0$ such that for all
$s>0$, we have 
\[
a_{x,y,U'}(s)\leq c_{3}'+\sum_{i=1}^{l}b_{x,y,U',\alpha_{i}V'}(s).
\]

We first notice that by definition we have 
\[
b_{x,y,U',\g V'}(s)=b_{x,\g^{-1}y,U',V'}(s).
\]
 It is because $\left(B_{x,\g^{-1}y,U',V}(s)\right)\g^{-1}$ $=B_{x,y,U',\g V'}(s)$
and $d^{a,b}(x,g\cdot\g^{-1}y)=d^{a,b}(x,g\g^{-1}\cdot y)$ for all
$g\in B_{x,\g^{-1}y,U',V}(s)$. 

Pick $r={\rm max}\{d_{1}(y,\alpha_{i}y):\ 1\leq i\leq l\}$, we notice
that by the triangle inequality of $d_{1}$, we know if $d_{1}(x,\g\alpha_{i}^{-1}y)\leq s-r,$
then $d_{1}(x,\g y)\leq s.$ So, we have 
\[
B_{x,\alpha_{i}^{-1}y,U',V'}(s)\subset B_{x,y,B_{r}U',V'}
\]
where $B_{r}U'$ is the $r-$neighborhood of $U'$. 

Moreover, again by the triangle equality of $d^{a,b}$, we know
\begin{alignat*}{1}
\left|-ad_{1}(x,\g y)-bd_{2}(x,\g y)-(-ad_{1}(x,\g\alpha_{i}^{-1}x)-bd_{2}(x,\g\alpha_{i}^{-1}x))\right|\\
\leq & ad_{1}(y,\alpha_{i}^{-1}x)+bd_{2}(y,\alpha_{i}^{-1}x))\\
\leq & c_{4}'(x,y,a,b).
\end{alignat*}
Therefore, for $1\leq i\leq l$ and $s>r$, 
\[
b_{x,\alpha_{i}^{-1}y,U',V'}(s-r)\leq e^{c_{4}'}b_{x,y,B_{r}U',V'}(s).
\]

Lastly, pick $U''$ such that $B_{r}U''\subset U'$ then for $s>r$
\begin{align*}
a_{x,y,U'',V'}(s-r) & \leq c_{3}'+\sum_{i=1}^{l}b_{x,y,U'',\alpha_{i}V'}(s-r)\leq c_{3}'+le^{c_{4}'}\cdot b_{x,y,B_{r}U'',V'}(s)\\
 & \leq c_{3}'+le^{c_{4}'}\cdot b_{x,y,U',V'}(s)\leq c_{3}'+le^{c_{4}'}\cdot a_{x,y,U'}(s).
\end{align*}

Taking limit of the both side and using the above lemma, we have completed
the proof. 
\end{proof}

\begin{lem}
For any relatively compact open set $W\subset T^{1}S_{1}$, we have 

\[
\limsup_{s\to\infty}\frac{1}{s}\log Z_{W}^{a,b}(s)\leq\delta_{PPS}^{a,b}.
\]
 
\end{lem}

\begin{proof}
Let $T\widetilde{p}:T^{1}\H^{2}\to T^{1}S_{1}=\G_{1}\backslash\H$
be the projection. Since $W$ is relative compact, there exist a compact
set $K\subset\H^{2}$ such that $W\subset T\widetilde{p}(\pi^{-1}(K))$. 

\textbf{Claim: }For a fixed $x\in K$, there exists a constant $c>0$
such that 
\[
Z_{W}^{a,b}(s)\leq e^{c}G_{x,x}(s+2r)
\]
where $r$ is the diameter of $K$.

pf. To see that, first we notice that for all $s\geq0$ and $\lambda\in{\rm Per_{1}}(s)$
such that $\lambda\cap W\neq\emptyset$, there exists a hyperbolic
element $\g_{\lambda}\in\G_{1}$ such that its translation axis ${\rm Axe_{\g_{\lambda}}}$
meets $K$, it has translation length $l_{1}[\lambda]$, and $\forall$$y\in{\rm Axe}_{\g_{\lambda}}$,
the image by $T\widetilde{p}$ of the unit tangent vector at $y$
pointing towards $\g_{\lambda}y$ belongs to $\lambda$.

We remark that the number of these elements $\g_{\lambda}$ is at
least equal to the cardinality of the pointwise stabilizer of ${\rm Axe}_{\g_{\lambda}}$
(i.e., the multiplicity of $\lambda$).

Let $x_{\lambda}$ be the closest point to $x$ on ${\rm Axe_{\g_{\lambda}}}$.
We have $d_{1}(x,x_{\lambda})\leq r$, because $x\in K$ and ${\rm Axe}_{\g_{\lambda}}\cap K\neq\emptyset$.
Thus by the triangle inequality, we know 
\[
l_{1}[\lambda]\leq d_{1}(x,\g_{\lambda}x)\leq d_{1}(x,x_{\lambda})+d_{1}(x_{\lambda},\g_{\lambda}x_{\lambda})+d(\g_{\lambda}x,\g_{\lambda}x_{\lambda})\leq l_{1}[\lambda]+2r\leq s+2r.
\]

Moreover, 
\begin{alignat*}{1}
\left|-ad_{1}(x,\g_{\lambda}x)-bd_{2}(x,\g_{\lambda}x)-al_{1}[\lambda]-bl_{2}[\lambda]\right| & \leq\left|-d^{a,b}(x,\g_{\lambda}x)-(-d^{a,b}(x_{\lambda},\g_{\lambda}x_{\lambda})\right|\\
 & \leq\left|d^{a,b}(x,x_{\lambda})\right|<c_{5}'.
\end{alignat*}

Hence,

\[
Z_{W}^{a,b}(s)\leq e^{c_{5}'}G_{x,x}(s+2r).
\]

\end{proof}

\begin{lem}
For any relatively compact open set $W\subset T^{1}S_{1}$, we have
\[
\liminf_{s\to\infty}\frac{1}{s}\log Z_{W}^{a,b}(s)\geq\delta_{PPS}^{a,b}.
\]

\end{lem}

\begin{proof}
Let $v\in T^{1}\H$ such that $(v^{-},v^{+})\in\Lambda^{1}\times\Lambda^{1}\backslash{\rm diagonal}$
and $T\widetilde{p}(v)\in W$, and let $x=\pi(v)$. 

\textbf{Claim:} There exists a constant $c>0$ such that 
\[
Z_{W}^{a,b}(s)\geq c\cdot b_{x,x,U',V'}(s).
\]

pf. Firstly, using the standard arguments (cf. Lemma 2.8 \cite{Pollicott:2012ud}
or P.150-151 \cite{Ghys:1990ux}), there exist small neighborhoods
$U'$ and $V'$ in $\H\cup\vbdy\H$ of $v^{+}$ and $v^{-}$, respectively,
such that if $\g\in\G_{1}$ satisfying $\g x\in U'$ and $\g^{-1}y\in V'$,
then $\g$ is a hyperbolic element and $v$ is close to the translations
axis ${\rm Axe}_{\g}$. Also, if $v_{\g}$ is the unit tangent vector
at $x_{\g}$ pointing $\g x_{\g}$, then we know $p(v_{\g})\in W$
(recall $W$ is open). Note that $U'\cap\Lambda_{1}\neq\emptyset$
and $V'\cap\Lambda_{1}\neq\emptyset$. 

Let $\g\in\G_{1}$ such that $d_{1}(x,\g x)\leq s$, $\g x\in U'$
and $\g^{-1}x\in V'$. Since $\g$ is hyperbolic, the corresponding
orbit $\lambda_{\g}$ is a closed orbit, and its length satisfies
\[
d_{1}(x,\g x)-2\leq l_{1}[\lambda_{\g}]\leq d_{1}(x,\g x)\leq s.
\]
Similarly, there exists a constant $c''>0$ such that $d_{2}(x,\g x)-c''\leq l_{2}[\lambda_{\g}]\leq d_{2}(x,\g x)$. 

Thus, there exists $c'_{6}\geq0$ such that

\begin{alignat*}{1}
\left|al_{1}[\lambda_{\g}]+bl_{2}[\lambda_{\g}]-d^{a,b}(x,\g x)\right| & =\left|a(d_{1}(x_{\g},\g x_{\g})-d_{1}(x,\g x))+b(d_{2}(x_{\g},\g x_{\g})-d_{2}(x,\g x))\right|\\
 & \leq2a+bc''=c'_{6}
\end{alignat*}

Notice that because the number of times a closed geodesic passes close
to a given point is at most linear in its length, we know the cardinality
of the fibers of the map $\g\to\lambda_{\g}$ is at most $c'_{7}s$
for some constant $c'_{7}>0$. Hence, we have 

\[
Z_{W}^{a,b}(s)\geq\frac{e^{-c_{6}'}}{c_{7}'}b_{x,x,U',V'}(s).
\]

\end{proof}

\subsection*{The Proof of Theorem \ref{thm:t_ab=00003Ddelta_ab}}

We continue using the same assumption as in the above subsection.

\begin{defn*}
$\ $\begin{enumerate}[font=\normalfont]

\item$\overline{G}_{x,y}^{a,b}(s):=\#\{\g\in G:\ d^{a,b}(x,\g\cdot y)\leq s\}$;

\item$\overline{G}_{x,y,1}^{a,b}(s):={\displaystyle \#\{\g\in G:\ s-1<d^{a,b}(x,\g\cdot y)\leq s\}}$;

\item$\overline{A}_{x,y,U',s}:=\{\g\in G:\ d^{a,b}(x,\g y)\leq s\ {\rm and}\ \rho_{1}(\g)y\in U'\}$
where $U'$ is an open set in $\vbdy\H\times\H$;

\item$\overline{a}_{x,y,U'}(s):=\#\overline{A}_{x,y,U',s}$;

\item${\displaystyle \overline{B}_{x,y,U',V',s}:=\{\g\in G:\ d^{a,b}(x,\g y)\leq s,\ \rho_{1}(\g)y\in U'\ {\rm and\ }\rho_{1}(\g^{-1})x\in V'\}}$
where $U',V'$ are open sets in $\vbdy\H\times\H$; and

\item$\overline{b}_{x,y,U',V'}(s):=\#\overline{B}_{x,y,U',V',s}.$

\end{enumerate}
\end{defn*}

Recall that $\delta^{a,b}$ is the critical exponent of $Q^{a,b}(s)={\displaystyle \sum_{\g\in G}\exp(-s\cdot d^{a,b}(x,\g y))}$,
i.e., $Q^{a,b}(s)$ converges when $s>\delta^{a,b}$ and diverges
when $s<\delta^{a,b}$.

\begin{thm*}
[Theorem  \ref{thm:t_ab=00003Ddelta_ab}]we have 
\[
\delta^{a,b}=\lim_{s\to\infty}\frac{1}{s}\log\overline{G}_{x,y}^{a,b}(s)=\lim_{s\to\infty}\frac{1}{s}\log\overline{a}_{x,y,U'}(s)=\lim_{s\to\infty}\frac{1}{s}\log\overline{b}_{x,y,U',V'}(s)
\]

\end{thm*}

Let $\iota:\rho_{1}(G)\to\rho_{2}(G)$ be the boundary-preserving
isomorphism. For $\g\in\rho_{1}(G):=\G_{1}$ the weighted distance
can be written as $d^{a,b}(x,\g y)=d(x,\g y)+d(x,\iota(\rho_{1}(\g))y)$.
Therefore, the above growth rates can be equivalently defined as:
\begin{itemize}

\item$\overline{G}_{x,y}^{a,b}(s):=\#\{\g\in\G_{1}:\ d^{a,b}(x,\g\cdot y)\leq s\}$.

\item$\overline{G}_{x,y,1}^{a,b}(s):=\#\{\g\in\G_{1}:\ s-1<d^{a,b}(x,\g\cdot y)\leq s\}.$

\end{itemize}

The proof is given by the following lemmas.

\begin{lem}
\[
\liminf_{s\to\infty}\frac{1}{s}\log\overline{G}_{x,y,1}^{a,b}(s)\leq\delta^{a,b}\leq\limsup_{s\to\infty}\frac{1}{s}\log\overline{G}_{x,y,1}^{a,b}(s).
\]

\end{lem}

\begin{proof}
The proof of Lemma \ref{lem:infG-delta-supG} works here. One just
need to replace $Q_{PPS,x,y}^{a,b}(s)$ by $Q^{a,b}(s)$, $G_{x,y,1}^{a,b}(s)$
by $\overline{G}_{x,y,1}^{a,b}(s)$ and $\delta_{PPS}^{a,b}$ by $\delta^{a,b}$. 
\end{proof}

\begin{lem}
The above inequalities are indeed equalities. Moreover, 
\[
\delta^{a,b}=\lim_{n\to\infty}\frac{1}{n}\log\overline{G}_{x,y}^{a,b}(s).
\]

\end{lem}

\begin{proof}
This proof is identical with the proof of Lemma \ref{lem:orbit-growth-limit}.

We notice that by the triangle inequality, it is obvious that the
$\limsup_{s\to\infty}\frac{1}{s}\log\overline{G}_{x,y}^{a,b}(s)$
does not depend on the reference point $x$ and $y$. W.l.o.g., we
pick $x=y=o$. Recall the generating set of the extended Schottky
group $G=\pi_{1}S$ is ${\cal A}=\{h_{1},..,h_{N_{1}},p_{1},...,p_{N_{2}}\}$
with $N_{1}+N_{2}\geq3$. 

Let \begin{itemize}

\item$\overline{E}_{n}:=\{\g\in G:\ n-1<d^{a,b}(o,\g o)\leq n\}$.

\item $b_{n}:=\overline{G}_{x,y,1}^{a,b}(n)={\displaystyle \sum_{\g\in E_{n}}e^{-d^{a,b}(o,\g o)}}.$

\end{itemize}

By Lemma \ref{lem:super-multiplicativity}, it is enough prove that
there exist $M>0$ and $N\in\N$ such that for all $n,m\in\N$, we
have 
\[
b_{n}b_{m}\leq M\sum_{i=-N}^{i=N}b_{n+m+i}.
\]

\textbf{Claim: }There exist $N\in\N$ and $M>o$ such that $\#\overline{E}_{n}\times\#\overline{E}_{m}\leq M\cdot\sum_{i=-N}^{i=N}\#\overline{E}_{n+m+i}$

pf. Let $\g_{n}\in\overline{E}_{n}$ and $\g_{m}\in\overline{E}_{m}$,
by Lemma \ref{lem:d(x,y)>d(x,o)+d(y,0)+C}, there exists $\alpha\in{\cal A}$
such that 
\[
\left|d^{a,b}(o,\g_{n}\alpha\g_{m}o)-d^{a,b}(o,\g_{n}o)-d^{a,b}(o,\g_{m}o)\right|<aC_{1}+bC_{2}
\]
where $C_{1}$ only depending on $\rho_{1}$ and $C_{2}$ only depending
on $\rho_{2}$ (same $C_{1},C_{2}$ in the proof of Lemma \ref{lem:orbit-growth-limit}). 

Thus, 
\[
n+m-aC_{1}+bC_{2}-2<d^{a,b}(o,\g_{n}\alpha\g_{m}o)\leq n+m+aC_{1}+bC_{2}+2,
\]

Let us consider the map

\begin{align*}
\overline{\Psi}:\overline{E}_{n}\times\overline{E}_{m} & \to\sum_{i=-(C_{1}+aC_{1}+bC_{2}+2)}^{i=C_{1}+aC_{1}+bC_{2}+2}\#\overline{E}_{n+m+i}\\
(\g_{n},\g_{m}) & \mapsto\g_{n}\alpha\g_{m}
\end{align*}

This maps is obvious not one-to-one. Nevertheless, we claim $\#\overline{\Psi}^{-1}(\g_{n}\alpha\g_{m})$
is finite. By Lemma \ref{lem:not_too_far}, we know that for $i=1,2$,
$d(\rho_{i}(\g_{n})o,[o,\rho_{i}(\g_{n}\alpha\g_{m})o])\leq D'$,
which implies if there exist $\g_{n}'\in\overline{E}_{n}$ and $\g_{m}^{'}\in\overline{E}_{m}$
such that $\g'_{n}\alpha\g_{m}'=\g{}_{n}\alpha\g_{m}$ then $d^{a,b}(\g_{n}o,\g_{n}'o)\leq D''$
where $D'$ and $D''$ are constants only depending on $a,$ $b,$
$\rho_{1}$ and $\rho_{2}$. Moreover, by the discreteness of $\rho_{i}(\G)$,
the set $\{\g\in G:\ d^{a,b}(\g o,o)\leq2(D''+1)\}$ is finite (say,
smaller than or equal to $\overline{M}_{1}$). Hence $\#\overline{\Psi}^{-1}(\g_{n}\alpha\g_{m})\leq\overline{M}_{1}^{2}$.
Therefore, 
\[
\#\overline{E}_{n}\times\#\overline{E}_{m}\leq(N_{1}+N_{2})\overline{M}_{1}^{2}\cdot\sum_{i=-(aC_{1}+bC_{2}+2)}^{i=aC_{1}+bC_{2}+2}\#\overline{E}_{n+m+i}
\]
 where $N_{1}+N_{2}$ is the cardinality of ${\cal A}$. 

Moreover, we know 
\[
\left|d^{a,b}(o,\g_{n}\alpha\g_{m}o)-d^{a,b}(o,\g_{n}o)-d^{a,b}(o,\g_{m}o)\right|\leq aC_{1}+bC_{2},
\]
thus we have the lemma, more precisely, 
\[
b_{n}b_{m}\leq(N_{1}+N_{2})\overline{M}_{1}^{2}\cdot e^{aC_{1}+bC_{2}}\sum_{i=-(aC_{1}+bC_{2}+2)}^{i=aC_{1}+bC_{2}+2}b_{n+m+i}.
\]

\end{proof}

As presented in the proof of Lemma \ref{lem:all the same}, using
the following lemma one could prove e $\delta^{a,b}$ is also the
growth rate of closed geodesics. More precisely, one could prove that
the growth rate of $\overline{b}^{a,b}$ equal to the growth rate
of closed geodesics on $S_{1}$ and $S_{2}$.

\begin{lem}

\[
\delta^{a,b}=\lim_{s\to\infty}\frac{1}{s}\log\overline{a}_{x,y,U'}(s).
\]

\end{lem}

\begin{proof}
This proof is identical with the proof of Lemma \ref{lem:pps_a}.

It is clear that $\overline{G}_{x,y}^{a,b}(s)\geq\overline{a}_{x,y,U'}(s)$,
so it is enough the prove that 
\[
\liminf_{s\to\infty}\frac{1}{s}\log\overline{a}_{x,y,U'}(s)\geq\delta^{a,b}.
\]

For fixed $(x,y)\in\H$ and $U'$ an open set in $\partial_{\infty}\H\cup\H$,
since $\G_{1}$ is non-elementary every orbit in $\Lambda_{1}$ is
dense. Thus, by the compactness of $\Lambda_{1}$, there exist $\g_{1},\g_{2},...,\g_{k}$
such that $\Lambda_{1}\subset\bigcup_{i=1}^{k}\g_{i}U'$. Furthermore,
since ${\rm Conv}(\Lambda_{1})\backslash\bigcup_{i=1}^{k}\g_{i}U'$
is compact, there exists a constant $c_{1}$ such that 
\[
\overline{G}_{x,y}^{a,b}(s)\leq\sum_{i=1}^{k}\overline{a}_{x,y,\g_{i}U'}(s)+c_{1}.
\]

\textbf{Claim: }$\overline{a}_{x,y,\g_{i}U'}(s-r)\leq\overline{a}_{x,y,U'}(s)$
where $r=\max\{d^{a,b}(\g_{i}x,x)\}$. 

pf. First we notice that by definition we have $\overline{a}_{x,y,U'}(s)=\overline{a}_{\g x,y,\g U'}(s).$

Therefore, we have 
\[
\overline{a}_{x,y,\g U'}(s-r)=\overline{a}_{\g_{i}^{-1}x,y,U'}(s-r):=\#\{\g\in\G_{1}:\ d^{a,b}(\g_{i}^{-1}x,\g y)\leq s-r,\ {\rm and}\ \g y\in U'\}.
\]

Notice that since $d^{a,b}$ satisfies the triangle inequality, we
know if $d^{a,b}(\g_{i}^{-1}x,\g y)\leq s-r$ for some $s\in(r,\infty)$,
then $d^{a,b}(\g y,x)<s$. So, we have 
\begin{alignat*}{1}
\overline{A}_{\g^{-1}x,y,U',s-r} & =\{\g\in\G_{1}:\ d^{a,b}(\g_{i}^{-1}x,\g y)\leq s-r,\ {\rm and}\ \g y\in U'\}\\
\subset & \{\g\in\G_{1}:\ d^{a,b}(x,\g y)\leq s,\ {\rm and}\ \g y\in U'\}=A_{x,y,U',s,}.
\end{alignat*}
Hence, we have the claim. 

Using this claim, we have 
\[
\overline{G}_{x.x}^{a,b}(s-r)\leq k\cdot\overline{a}_{x,y,U'}(s)+c_{1},
\]

and the result follows. 
\end{proof}

\begin{lem}
\[
\delta^{a,b}=\lim_{s\to\infty}\frac{1}{s}\log\overline{b}_{x,y,U',V'}(s).
\]

\end{lem}

\begin{proof}
This proof is the same one as the proof of Lemma \ref{lem:pps_b}.

We first notice that there exist $\alpha_{1},...,\alpha_{l}$ such
that $\Lambda_{1}\subset\bigcup_{i=1}^{l}\alpha_{i}V'$, and because
${\rm Conv}(\Lambda_{1})\backslash\bigcup_{i=1}^{l}\alpha_{i}V'$
is compact, we have
\[
\overline{a}_{x,y,U'}(s)\leq c_{2}+\sum_{i=1}^{l}\overline{b}_{x,y,U',\alpha_{i}V'}(s).
\]

We first notice that by definition we have 
\[
\overline{b}_{x,y,U',\g V'}(s)=\overline{b}_{x,\g^{-1}y,U',V'}(s).
\]
Then we pick $r={\rm max}\{d^{a,b}(y,\alpha_{i}y):\ 1\leq i\leq l\}$.
Notice that by the triangle inequality of $d^{a,b}$, we know if $d^{a,b}(x,\g\alpha_{i}^{-1}y)\leq s-r,$
then $d^{a,b}(x,\g y)\leq s.$ So, we have 
\[
\overline{B}_{x,\alpha_{i}^{-1}y,U',V'}(s)\subset\overline{B}_{x,y,B_{r}U',V'}
\]
where $B_{r}U'$ is the $r-$neighborhood of $U'$. 

Moreover, again by the triangle equality of $d^{a,b}$, we know
\begin{alignat*}{1}
\left|-ad_{1}(x,\g y)-bd_{2}(x,\g y)-(-ad_{1}(x,\g\alpha_{i}^{-1}x)-bd_{2}(x,\g\alpha_{i}^{-1}x))\right|\\
\leq & ad_{1}(y,\alpha_{i}^{-1}x)+bd_{2}(y,\alpha_{i}^{-1}x))\\
\leq & c_{4}'(x,y,a,b).
\end{alignat*}
Therefore, for $1\leq i\leq l$ and $s>r$, 
\[
\overline{b}_{x,\alpha_{i}^{-1}y,U',V'}(s-r)\leq e^{c_{4}'}\overline{b}_{x,y,B_{r}U',V'}(s).
\]

Lastly, pick $U''$ such that $B_{r}U''\subset U'$, then for $s>r$
\begin{align*}
\overline{a}_{x,y,U'',V'}(s-r) & \leq c_{2}+\sum_{i=1}^{l}\overline{b}_{x,y,U'',\alpha_{i}V'}(s-r)\leq c_{2}+le^{c_{4}'}\cdot\overline{b}_{x,y,B_{r}U'',V'}(s)\\
 & \leq c_{2}+le^{c_{4}'}\cdot\overline{b}_{x,y,U',V'}(s)\leq c_{2}+le^{c_{4}'}\cdot\overline{a}_{x,y,U'}(s).
\end{align*}
Taking limit of the both side and using the above lemma, we have completed
the proof. 
\end{proof}